\documentclass[11pt]{article} %
\usepackage{fullpage}
\usepackage{amsthm}
\usepackage{amsmath, amssymb, graphicx, url}
\usepackage{color}
\usepackage{todonotes}
\usepackage{dsfont}
\usepackage{hyperref}
\usepackage[capitalize,noabbrev]{cleveref}

\newcommand{\cH}{\mathcal{H}}

\newcommand{\cN}{\mathcal{N}}
\newcommand{\cO}{\mathcal{O}}
\newcommand{\cP}{\mathcal{P}}

\newcommand{\E}{\mathbb{E}}

\newcommand{\N}{\mathbb{N}}
\newcommand{\Q}{\mathbb{Q}}
\newcommand{\R}{\mathbb{R}}

\newcommand{\xsingle}{\widehat\rho}
\newcommand{\xell}{\widehat\rho^\ell}

\newcommand{\xellL}{\widehat\rho^L}
\newcommand{\xellzero}{\widehat\rho^0}
\newcommand{\xelltilde}{\widetilde{\widehat\rho}^\ell}
\newcommand{\xelltildeminus}{\widetilde{\widehat\rho}^{\ell-1}}

\newcommand{\mfsingle}{\rho}
\newcommand{\mfell}{\rho^\ell}
\newcommand{\mfellminus}{\rho^{\ell-1}}
\newcommand{\mfellL}{\rho^L}
\newcommand{\mfellzero}{\rho^0}
\newcommand{\mirrormfsingle}{\overline\rho}
\newcommand{\mirrormfell}{\overline\rho^\ell}

\newcommand{\mirrormfelltilde}{\widetilde{\overline\rho}^\ell}
\newcommand{\mirrormfelltildeminus}{\widetilde{\overline\rho}^{\ell-1}}
\newcommand{\mirrormfellzero}{\overline\rho^0}

\usepackage{enumitem}
\usepackage[normalem]{ulem}

\newcommand{\dd}{{\rm d}}

\DeclareMathOperator{\Var}{Var}
\DeclareMathOperator{\var}{Var}
\DeclareMathOperator{\KL}{KL}
\DeclareMathOperator{\SL}{SL}
\DeclareMathOperator{\ML}{ML}

\DeclareMathOperator{\clip}{c_{lip}}
\DeclareMathOperator{\alip}{a_{lip}}
\DeclareMathOperator{\blip}{b_{lip}}
\DeclareMathOperator{\dlip}{d_{lip}}
\DeclareMathOperator{\cost}{cost}

\newcommand{\Capp}{{C_{\rm app}}}

\usepackage{todonotes}

\definecolor{sw}{rgb}{1,0.53,0.0}

\definecolor{jz}{rgb}{0.1,0.45,0.1}

\DeclareMathOperator{\diag}{diag}

\def\KL#1#2{\textnormal{KL}({#1}\Vert{#2})}

\usepackage{mathtools}
\DeclarePairedDelimiter\ceil{\lceil}{\rceil}

\usepackage{algorithm,algpseudocode} 
\usepackage{authblk}

\title{On the mean-field limit for Stein variational gradient descent: stability and multilevel approximation}

\author[1]{Simon Weissmann}
\author[2]{Jakob Zech}
\date{\today}
\affil[1]{\normalsize
  Universit\"at Mannheim, Institute of Mathematics\\
  68138 Mannheim, Germany\\
\texttt{simon.weissmann@uni-mannheim.de}
}
\affil[2]{\normalsize
  Universit\"at Heidelberg, Interdisziplin\"ares Zentrum f\"ur Wissenschaftliches Rechnen\\
  69120 Heidelberg, Germany\\
\texttt{jakob.zech@uni-heidelberg.de}
}

\date{\today}

\newtheorem{theorem}{Theorem}[section]

\newtheorem{lemma}[theorem]{Lemma}

\newtheorem{proposition}[theorem]{Proposition}
\newtheorem{remark}[theorem]{Remark}

\newtheorem{assumption}[theorem]{Assumption}

\begin{document}

\maketitle

\begin{abstract}
In this paper we propose and analyze a novel multilevel version of Stein variational gradient descent (SVGD). SVGD is a recent particle based variational inference method. For Bayesian inverse problems with computationally expensive likelihood evaluations, the method can become prohibitive as it requires to evolve a discrete dynamical system over many time steps, each of which requires likelihood evaluations at all particle locations. To address this, we introduce a multilevel variant that involves running several interacting particle dynamics in parallel corresponding to different approximation levels of the likelihood. By carefully tuning the number of particles at each level, we prove that a significant reduction in computational complexity can be achieved. As an application we provide a numerical experiment for a PDE driven inverse problem, which confirms the speed up suggested by our theoretical results.
\end{abstract}


\section{Introduction}\label{sec:Introduction}

The Stein variational gradient descent (SVGD) method is an optimization-based variational inference algorithm and has been introduced as an efficient sampling method for Bayesian inference problems \cite{SVGD}. While in practical applications the algorithm is implemented through a finite particle approximation, the theoretical understanding has mainly been developed in the mean-field (MF) limit, e.g.\ \cite{SVGDgradientflow}.
Therefore, it is crucial to develop efficient approximations of the MF limiting system. In this manuscript, we view the interacting particle system as a approximation of the MF limit and develop a novel multilevel SVGD (ML-SVGD) algorithm. %
Our analysis primarily relies on the finite time convergence analysis
  of SVGD developed in \cite{NEURIPS20203202111c}, and on concepts
  from multilevel Monte Carlo (MLMC), see for example \cite{MR2436856,Heinrich2001}.
We point out that MLMC methods have %
in earlier works been combined with Markov chain Monte Carlo (MCMC)
methods \cite{MLMCMC} and deterministic quadrature schemes \cite{MR2648461,MR3636617,MR3502561,MLGPC}. 

As a motivating example, we consider SVGD for solving a Bayesian
inference problem. In Bayesian inverse problems \cite{DS2017} the goal
is to explore a posterior distribution through the generation of
samples.
Apart from MCMC methods \cite{RC2004} which are widely used,
methods rooted in variational inference have become %
a popular alternative, e.g.\ \cite{doi:10.1080/01621459.2017.1285773}.
We consider the observation model
\[Y = F(X) + \eta, \]
where $F:\R^d\to \R^{n_y}$ denotes the forward model, $(X,Y)$ are assumed to be jointly varying random variables on $\R^d\times \R^{n_y}$ and $\eta\sim\mathcal N(0,C_0)$ denotes Gaussian additive observational noise that is assumed independent of $X$. %
For a prior distribution $\mathbb Q_0$ on $X$, the central task in Bayesian inference is to quantify the posterior distribution---the distribution of $X$ conditional on a realization $Y=y$---which can be written as
\begin{equation*} 
\begin{split}
\pi(x) &\propto \exp(-\frac12\Phi(x,y))\Q_0(dx),\\
 \Phi(x,y)&:= \|C_0^{-1/2}(F(x)-y)\|^2,\quad x\in\R^d,\ y\in\R^{n_y}\, . 
 \end{split}
 \end{equation*}
When the prior is assumed to admit a Lebesgue density, the posterior probability density function (pdf) %
is often written in the form
\[\pi(x) \propto \exp(-V(x)),\quad x\in\R^d\, , \]
for a potential
$V:\R^d\to\R$.
One way to quantify the posterior distribution is to solve the variational problem
\[\min_{\psi\in\Psi}\ \KL{\psi}{\pi} \]
where $\Psi$ denotes a family of (tractable) probability distributions and $\KL{\psi}{\pi}$ denotes the Kullback-Leibler (KL) divergence of $\psi$ and $\pi$. In the mean-field limit, SVGD can be motivated as Euler approximation of the Wasserstein gradient flow represented in a reproducing kernel Hilbert space (RKHS) minimizing the KL divergence between a reference distribution and target distribution of interest \cite{Duncan2019OnTG,NEURIPS20203202111c,SVGDgradientflow}. %
In discrete-time it can be viewed as an iterative scheme which involves multiple evaluations of the gradient $\nabla \log\pi$, where $\pi$ is the target pdf of the posterior. %
Each evaluation of $\nabla\log\pi$ requires to evaluate the forward map $F$, which can be computationally expensive. This is in particular the case if $F$ models some physical, chemical or biological phenomenon, that requires to numerically solve a partial differential equation. In these scenarios where we are only able to evaluate an approximation $F_\ell$ of $F$, where $\ell\in\N$ stands for a discretization level that is associated with the accuracy of the approximation.
  A convergence analysis for SVGD as an interactive particle
  system needs to take into account
  such errors stemming from the approximation
  of $F$ as well as the finite number of particles.
  Viewing the interacting particle system as
  an approximation of the mean field limiting system, %
  in this paper we propose
  a multi-level variant of SVGD in the spirit of MLMC. %
  More precisely, we formulate a novel family of independent particle systems, where many particles evolve according to a dynamics driven by $F_\ell$ with low $\ell$ (associated to low accuracy but also low computational cost) which is corrected by few particles evolving with a dynamics driven by $F_\ell$ for high $\ell$ (associated to high accuracy and high computational cost). This allows us to keep the overall computational cost low and thereby speed up the algorithm.

  We emphasize that our %
  method differs crucially from the previously introduced multi-level SVGD method in \cite{alsup22a}, %
  which is based on gradually increasing the accuracy level %
  as the particles evolve.
  Their
  algorithm is formulated in the mean-field limit and the proposed method employs
 an equal number of particles on each accuracy level. In contrast to our work, their method %
   strongly depends on assuming exponential convergence for the mean-field limiting system towards the target distribution. %
Ideas similar to \cite{alsup22a} have also been proposed as general
\textit{multi-level optimization} tools, see 
\cite{MR4284423,MR4294188} specifically for stochastic gradient methods
and \cite{MLOPTI} for a unified approach treating various deterministic
and stochastic methods.

In order to formulate our multi-level SVGD scheme we borrow ideas from multi-level particle methods in the area of data assimilation \cite{LSZ2015}. The idea of viewing the particle systems as Monte Carlo (MC) approximation of the mean-field limiting system has led to the formulation of various multi-level ensemble Kalman filters \cite{doi:10.1137/15M100955X,CHLNT2021,HST2020,chada2021multilevel} and more generally of multi-level mean-field approximation of McKean-Vlasov equation \cite{HT2018}. 

\paragraph{Contributions:}  
\begin{itemize}
\item 
We propose a
  novel multilevel SVGD method that is based on a careful combination
  of several finite interacting particle systems with differing
  sample sizes and differing accuracy levels.  
\item In order to analyse the multilevel SVGD method we study the
  behavior of the MF system under changes in the target probability
  distribution function (pdf) $\pi$. More precisely, we prove
  that %
  small changes in %
      $\pi$ yield %
        small changes in the solution to the MF system in terms
        of the Wasserstein-2 distance.
  \item %
    We provide a complete error analysis of the proposed multilevel
    estimator for expectations with respect to the MF
    solution.  %
    As we show, a careful tuning of the required samples at each level
    allows to decrease the overall computational cost of the
    algorithm.
\end{itemize}

\subsection{Preliminaries and notation}
In this manuscript, we focus on the
  discrete-time formulation of SVGD, for which a convergence analysis was
  recently developed in \cite{NEURIPS20203202111c}. Our analysis
  strongly builds upon these results, and we therefore largely adopt
  their setting and notation, and also refer to this paper for more
  details on the operators introduced in the following.

The MF limit of SVGD can be described in terms of the Wasserstein distance
between the empirical measure over the ensemble of particles and the
limiting distribution evolving in time. %
Throughout the following we consider $\R^d$ equipped with the Borel
$\sigma$-algebra $\mathcal{B}(\R^d)$. For two probability measures $\mu$, $\eta$ on $\R^d$,
the Wasserstein $p$-distance is defined as
\begin{equation*}
  \mathcal{W}_p(\mu,\eta) = \inf_{\nu\in \mathcal P(\mu,\eta)}\Bigg(\int_{\R^d\times\R^d} d(x,y)^p\, {\mathrm d}\nu(x,y)\Bigg)^{{1/p}},
\end{equation*}
where %
$\mathcal P(\mu,\eta)$ denotes the set of all couplings of $\mu$ and $\eta$; i.e.~for $\nu\in\mathcal P(\mu,\eta)$ we have
\begin{align*} 
\int_{\R^d\times B}{\mathrm d}\nu(x,y) &= \eta(B),\quad B\in\mathcal{B}(\R^d),\\ \int_{B\times \R^d}{\mathrm d}\nu(x,y)&=\mu(B),\quad B\in\mathcal{B}(\R^d). 
\end{align*}

Let $\cP_2(\R^d)$ be the set of probability measures on $\R^d$ with finite
first and second moment. For $\mu\in \cP_2(\R^d)$, %
we denote its mean by $m(\mu)$ and its variance by $\var(\mu)$,
\[m(\mu) = \int_{\R^d} x\,\mu(\dd x),\quad \var(\mu) = \int_{\R^d} \|x-m(\mu)\|^2 \mu(\dd x)\, . \]

To formulate SVGD we next recall some notation and
  operators introduced in \cite{NEURIPS20203202111c}, and also refer
  to this paper for further details. Throughout we
consider a fixed stationary and positive definite kernel
$k:\R^d\times\R^d \to \R$, i.e.\ $k(x,y)$ depends only on $x-y$.
We denote the corresponding RKHS by
$\cH_0$ with inner product $\langle \cdot, \cdot \rangle_{\cH_0}$.
The function space $\mathcal H = \{f:\R^d\to \R^n \mid f = (f_1,\dots,f_n),\ f_i\in\mathcal H_0\}$ defines the product RKHS $\mathcal H$ with inner product $\langle f,g\rangle_{\mathcal H} = \sum_{i=1}^n\langle f_i,g_i\rangle_{\mathcal H_0}$.
We denote the derivative of the kernel with respect to its first component by $\nabla_1 k(x,y)$, and similarly with respect to its second component by $\nabla_2 k(x,y)$. Given a probability measure $\mu\in\cP_2(\R^d)$ and $\int_{\R^{d}} k(x,x)\, \mu(\dd x)<\infty$, i.e.~it holds that $\cH\subset L^2(\mu)$, %
and we let $S_\mu:L^2(\mu)\to\cH$ via
\[S_\mu f(y) := \int_{\R^d} k(x,y)f(x)\,\mu(\dd x),\quad f\in L^2(\mu),\ y\in \R^d\,.\]
Using %
the natural embedding
$\iota:\cH \to L^2(\mu)$, we %
define $P_\mu:L^2(\mu)\to L^2(\mu)$ %
as $P_\mu := \iota S_\mu$. The operator $P_\mu$ applied to the functional $f(\cdot) = \nabla \log \frac{\mu}{\pi}(\cdot)$ is formally of the form
\begin{equation*} 
P_\mu \nabla \log \frac{\mu}{\pi}(y)= -\int_{\R^d} \left( \nabla
    \log \pi(x) k(x,y) + \nabla_1 k(x,y)\right)\, \mu(\dd x) 
  \end{equation*}
for $y\in\R^d$. We consider this to be a definition of the
  expression on the left-hand side. The formula can be motivated
  using integration by parts if all functions and measures
  are sufficiently smooth, \cite{NEURIPS20203202111c}.
In the specific case where $\widehat\mu$ is an empirical measure $\widehat\mu = \frac{1}{N} \sum_{j=1}^N\delta_{x^{(j)}}$, we thus have for $y\in\R^d$
\begin{equation*}
 P_{\widehat \mu} \nabla \log \frac{\widehat\mu}{\pi}(y)= -\frac1N \sum_{j=1}^N \left(\nabla \log\pi(x^{(j)}) k(x^{(j)},y) + \nabla_1 k(x^{(j)},y) \right).
 \end{equation*}

\section{Stein variational gradient descent}\label{sec:SVGD}
The algorithm of SVGD %
was originally introduced as a finite interacting particle system
  $\{X_n^{(j)}(\cdot),j=1,\dots,N\}$, $n\ge 0$, of ensemble size
  $N\ge2$ evolving through
\begin{equation}\label{eq:SVGD_discr}
\begin{split}
    X_{n+1}^{(i)} = X_n^{(i)} - \gamma \frac{1}{N}\sum_{j=1}^N\Bigg\{k\big(X_n^{(j)},X_n^{(i)}\big)\nabla \log\pi\big(X_n^{(j)}\big)+ \nabla_1 k\big(X_n^{(j)},X_n^{(i)}\big) \Bigg\}
 \end{split}
\end{equation}
for $i=1,\dots,N$ with i.i.d.~initialization $X_0^{(i)} \sim \mfsingle_0$ for the
initial distribution $\mfsingle_0$. Here and throughout the rest of this
manuscript we denote by $\gamma>0$ a fixed step size. Moreover, throughout the paper we consider an initial distribution 
\begin{equation} \label{eq:initial_moment}
  \mfsingle_0\in \mathcal P_2(\R^d).
\end{equation}

We now introduce notation for three different stochastic dynamical
  systems which allow us to analyse the behaviour of SVGD in the
  following.
  \begin{enumerate}
  \item {\bf Interacting particles:} The particle system generated by
    \eqref{eq:SVGD_discr}
  will in the following be denoted by
  $\mathcal X_n = \{X_n^{(j)},\ j=1,\dots,N\}$. The corresponding
  empirical measure over the particle system is
  $\xsingle_n = \frac{1}{N}\sum_{j=1}^N \delta_{X_n^{(j)}}$. We
  introduce the notation
  \begin{align*} 
  R_\rho(z) &:= P_{\rho} \nabla\log\left(\frac{\rho}{\pi}\right)(z)\\ &=
    -\int_{\R^d} k(x,z)\nabla_x\log\pi(x) + \nabla_1
    k(x,z)\,\rho({\mathrm d}x)
    \end{align*}
     and also write
  \eqref{eq:SVGD_discr} as
\begin{equation}\label{eq:SVGD_interacting}
\begin{split}
    X_{n+1}^{(i)} &= X_n^{(i)} - \gamma P_{\xsingle_n}\nabla\log\left(\frac{\xsingle_n}{\pi}\right)(X_n^{(i)})\\ &= X_n^{(i)} -\gamma \widehat R_n(X_n^{(i)}),\quad \widehat R_n(\cdot):=R_{\xsingle_n}(\cdot),
    \end{split}
  \end{equation}
  with i.i.d.~initalization $X_0^{(i)}\sim\rho_0$, $i=1,\dots,N$.
  One important property used in the
  following %
  is that the $X_n^{(j)}$ are identically distributed (but not
  independent) for $j=1,\dots,N$. We will thus often use
  $X_n^{(1)}$ as a representant.
  \item {\bf Mean field:} The dynamics
 \eqref{eq:SVGD_discr} can be viewed as Monte Carlo-like particle approximation to the MF system
\begin{equation}\label{eq:SVGD_mf_discr}
\begin{split}
    Z_{n+1} &= Z_n - \gamma P_{\mfsingle_n}\nabla\log\left(\frac{\mfsingle_n}{\pi}\right)(Z_n)\\ &= Z_n - \gamma R_n(Z_n),\quad Z_0\sim\mfsingle_0,
    \end{split}
\end{equation}
where $\mfsingle_n$ denotes the law of the random variable $Z_n$ and $R_n(z) := R_{\mfsingle_n}(z)$. The sequence of distributions $(\mfsingle_n)_{n\ge0}$ solves the MF equation 
\begin{equation}\label{eq:SVGD_mf_measure}
\begin{split}
    \mfsingle_{n+1} &= (I-\gamma R_n)_\sharp\rho_n\\ &= \left(I-\gamma P_{\mfsingle_n}\nabla\log\left(\frac{\mfsingle_n}{\pi}\right)\right)_{\sharp} \mfsingle_n\, ,
    \end{split}
\end{equation} 
where $T_{\sharp} \mu$ denotes the pushed forward measure of $\mu$ under a measurable mapping $T:\R^d\to\R^d$. %
Due to $\rho_0\in \mathcal P_2(\R^d)$,
according to \cite[Lemma~12]{NEURIPS20203202111c} %
\begin{equation}\label{eq:moment_iterate}
\rho_n\in\mathcal P_2(\R^d)
\end{equation}
for every %
$n\in\N$.
  \item {\bf Mirrored mean field:} In order to analyze the MF limit we will consider the auxiliary particle system $\mathcal Z_n = \{Z_n^{(j)},\ j=1,\dots,N\}$, $n\in\N$ with empirical measure $\mirrormfsingle_n = \frac{1}{N}\sum_{j=1}^N \delta_{Z_n^{(j)}}$, which mirrors the MF system 
\begin{equation}\label{eq:SVGD_mf_iid}
\begin{split}
    Z_{n+1}^{(j)} &= Z_n^{(j)} - \gamma P_{\mfsingle_n}\nabla\log\left(\frac{\mfsingle_n}{\pi}\right)(Z_n^{(j)})\\ &= Z_n^{(j)} - \gamma R_n(Z_n^{(j)}),\quad j=1,\dots,N,
    \end{split}
\end{equation}
with i.i.d.~initialization
$Z_0^{(1)}\sim\mfsingle_0$. %
Note that by definition $\mathcal Z_n$ %
is an i.i.d.~sample of %
$\rho_n$, and in particular each $Z_n^{(j)}$, $j=1,\dots,N$, is a
random variable with the same distribution as $Z_n$.
\end{enumerate}

 Throughout this paper, we consider $(\mathcal X_n)_{n\ge0}$ and $(\mathcal Z_n)_{n\ge0}$ as stochastic processes on the same probability space $(\Omega,\mathcal F,\mathbb P_0)$ under which $X_0^{(j)}(\omega) = Z_0^{(j)}(\omega)$ for %
 all $\omega\in\Omega$. The expectation with respect to $\mathbb P_0$ is denoted by $\mathbb E_0$ and for $p>0$ we define the space
 \begin{align*}
   L_0^p(\R) &:= L^p(\Omega,\mathcal F,\mathbb P_0; \R,\mathcal B(\R))\nonumber\\
   &:= \Big\{ f: \Omega\to\R \mid f\ \text{measurable},\ \int_{\R} |f|^p\,
   \dd \mathbb P_0 < \infty\Big\}
\end{align*}
   with norm
 $\|f\|_{L_0^p(\R)} := \left(\int_{\R} |f|^p\, \dd \mathbb
   P_0\right)^{1/p} = \E_0[|f|^p]^{1/p}$. We summarize the introduced
 systems in \Cref{table:1}.

 \begin{table}[htb!]
   \begin{center}
\begin{tabular}{|l|c|c|c|c|c|}
  \hline
   & Notation & Evolution & Initialization & Distribution  & Size\\
  \hline
   IP & $(\mathcal X_n,\xsingle_n)$ & \eqref{eq:SVGD_interacting} & $X_0^{(j)}\sim \rho_0$ & $\xsingle_n = \frac{1}{N}\sum_{j=1}^N \delta_{X_n^{(j)}}$  & N\\
  MF & $(Z_n,\mfsingle_n)$ & \eqref{eq:SVGD_mf_discr} & $Z_0 \sim \rho_0$ & $Z_n\sim\mfsingle_n$  & 1 \\
  MMF  & $(\mathcal Z_n,\mirrormfsingle_n)$ & \eqref{eq:SVGD_mf_iid} & $Z_0^{(j)}=X_0^{(j)}$ & $\mirrormfsingle_n = \frac{1}{N}\sum_{j=1}^N \delta_{Z_n^{(j)}}$  & N \\
  \hline
\end{tabular}\caption{Overview of the of the three particle dynamics introduced
    in Section \ref{sec:SVGD}: mean field (MF), mirrored mean field
    (MMF) and interacting particle (IP).}\label{table:1}
\end{center}
\end{table}

 In \cite{NEURIPS20203202111c} the authors consider a theoretical analysis of the MF limit of the discret-time system \eqref{eq:SVGD_discr} to \eqref{eq:SVGD_mf_discr} under the following assumptions.
\begin{assumption}\label{ass:kernel2}
  There exist finite and positive constants $M$, $C_V$, $B$ and $\blip$ such that:
    \begin{itemize}
    \item[A1] %
      The Hessian of $V = \log\pi$ is %
      uniformly bounded, i.e.\
      $\|\nabla^2V(x)\|\le M$ for all $x\in\R^d$.
  \item[A2] %
    The
    gradient of %
    $V=-\log\pi$ is uniformly bounded, %
    i.e.\ 
    $\|\nabla \log\pi(x)\| = \|\nabla V(x)\| \le C_V$ for all $x\in\R^d$. %
    \item[A3] %
      The kernel function $k:\R^d\times \R^d \to\R$ is a bounded, i.e.\
      $\|k(x,\cdot)\|_{\mathcal H}\le B$ and $\|\nabla_1 k(x,\cdot)\|_{\mathcal H}\le B$ for all $x\in\R^d$.
    \item[A4] %
      The
      kernel function $k:\R^d\times \R^d \to\R$ is continuously differentiable, Lipschitz-continuous and has Lipschitz-gradient, %
    \begin{align*}
        |k(x,x')-k(y,y')| &\le \blip(\|x-y\|+\|x'-y'\|),\\ \|%
      \nabla k(x,x')-%
      \nabla
      k(y,y')| &\le \blip(\|x-y\|+\|x'-y'\|)\, .
    \end{align*}
    \end{itemize}
\end{assumption}

The following Lipschitz property is one key-tool of proving the convergence towards the MF limit.

\begin{lemma}[Lemma~14 in \cite{NEURIPS20203202111c}]\label{lem:Lipschitz_drift}
  Under Assumption~\ref{ass:kernel2}
  there exists $\clip<\infty$ depending on the constants in Assumption~\ref{ass:kernel2}, such that
  the mapping $(z,\rho)\mapsto P_\rho\nabla\log\left(\frac{\rho}{\pi}\right)$ is $\clip$-Lipschitz in the sense %
    \begin{equation*} 
      \|P_{\rho_1}\nabla\log\left(\frac{\rho_1}{\pi}\right)(z_1)-P_{\rho_2}\nabla\log\left(\frac{\rho_2}{\pi}\right)(z_2)\|\le \clip [\|z_1-z_2\| + \mathcal W_2(\rho_1,\rho_2)]\, . 
    \end{equation*}
\end{lemma}
Moreover, we will make use of a discrete Gronwall inequality.
\begin{lemma}[Lemma~13 in \cite{NEURIPS20203202111c}]\label{lem:discr_Gr}
    Suppose that the real valued sequence $(c_n)_{n\in\N}$ satisfies $c_0=0$ and the iterative inequality
\[c_{n+1}\le (1+\gamma A)c_n + b\]
for some $\gamma$, $A$, $b>0$. Then $c_n$ satisfies 
\[c_{n}\le \frac{b}{\gamma A} \left(\exp(n\gamma A)-1\right)\,.\] 
\end{lemma}

Using both of these key-tools, the MF limit can be quantified in the following way:
\begin{proposition}[Proposition~7 in \cite{NEURIPS20203202111c}]\label{prop:MF}
Suppose Assumption~\ref{ass:kernel2} is satisfied. Then for all $T>0$ and any $n<T/\gamma$
\begin{align*} 
\mathbb E[\mathcal W_2(\xsingle_n,\mfsingle_n)] &\le c_n := \left(\frac1N\sum_{j=1}^N\mathbb E_0[\|X_n^{(j)}-Z_n^{(j)}\|^2]\right)^{1/2}\\ &\le\frac12\left(\frac{1}{\sqrt{N}} \sqrt{\Var(\mfsingle_0)} e^{A T}\right)(e^{2 A T}-1),
\end{align*}
where $A>0$ is a constant depending on $\pi$ and $k$.
\end{proposition}

\section{Multilevel Stein Variational gradient descent}\label{sec:MLSVGD}

In the following section, we propose a novel multilevel SVGD approach by applying ideas of MLMC methods for approximating expectations w.r.t.~the MF limit. We start the discussion by the viewpoint of standard single level approximations. 

\subsection{Single level approximation}\label{sec:MC_approx}
We aim to construct an estimator of the expectation over some functional of interest $\varphi:\R^d\to\R$ w.r.t.~$\mfsingle_n$, denoted by
\[\mfsingle_n[\varphi] := \int_{\R^d}\varphi(x)\,\mfsingle_n({\mathrm d}x)\, ,\]
where $(\rho_n)_{n\ge0}$ evolves through the MF equation \eqref{eq:SVGD_mf_measure}. We %
work under
the following assumption on $\varphi$.
\begin{assumption}\label{ass:phi}
    Let $\varphi:\R^d\to \R$ be Lipschitz-continuous with $\alip>0$ such that
    $|\varphi(x)-\varphi(y)|\le \alip \|x-y\|$.
\end{assumption}
Given %
$\varphi$, we define %
an the estimator of $\mfsingle_n[\varphi]$ by
\begin{equation}\label{eq:MC}
  \xsingle_n[\varphi] = \int_{\R^d} \varphi(x)\,\xsingle_n({\mathrm d}x) = \frac1{N}\sum_{i=1}^N \varphi(X_n^{(i)}),
\end{equation}
where $\xsingle_n = \frac1N\sum_{i=1}^N \delta_{X_n^{(i)}}$ and $\{X_n^{(j)},i=1,\dots,N\}_{n\ge0}$ evolves through \eqref{eq:SVGD_discr} with i.i.d.~initial ensemble $X_0^{(i)}\sim\mfsingle_0$, $i=1,\dots,N$. The convergence of the proposed estimator can be verified through the MF limit presented in Proposition~\ref{prop:MF}.

\begin{proposition}\label{prop:err_MC}
  Let $\varphi$ satisfy Assumption~\ref{ass:phi}. 
    Then %
    for all $n\le T/\gamma$
    \[\|\xsingle_n[\varphi]-\mfsingle_n[\varphi]\|_{L_0^2(\R)}=\mathbb E_{0}[ |\xsingle_n[\varphi]-\mfsingle_n[\varphi]|^2]^{\frac12}\le \frac{C}{\sqrt{N}},\]
    where $C>0$ is a constant depending on $\pi$, $k$, $\alip$ and $T$ but independent of $N$.
\end{proposition}

\subsection{Multilevel approximation}

Implementing the iterative scheme \eqref{eq:SVGD_discr} %
requires %
repeated evaluation of
$V = -\log \pi$ for the target measure $\pi$. However, in certain
applications only numerical approximations $\nabla \log \pi_\ell$ or
$\pi_\ell$ to $\nabla \log\pi$ and $\pi$ respectively are
available. %
Here $\ell\in\N_0$ is the
``level'' of the approximation,
which is assumed to be associated with its accuracy:
the larger $\ell$ the better the
approximation.
To be more precise, we work under the following
assumption. Without loss of generality and to
  keep the notation succinct, we use the same
  notation for the occurring
constants as in Assumption \ref{ass:kernel2}.
\begin{assumption}\label{ass:level_accuracy}
  There exist finite and positive constants $M$, $C_V$, $\beta$,
    $\Capp$ such that:
    \begin{enumerate}
    \item[B1] %
      The Hessian of $V_\ell = -\log\pi_\ell$ is %
      uniformly bounded, i.e.\ $\|\nabla^2V_\ell(x)\|\le M$
      for all $x\in\R^d$ and all $\ell\in\N_0$.
  \item[B2] %
    The gradients of %
    $V_\ell = -\log \pi_\ell$
    are uniformly bounded, i.e.\
    $\|\nabla \log\pi_\ell(x)\| = \|\nabla V_\ell(x)\| \le C_V$ for all $x\in\R^d$ and all $\ell\ge0$.
  \item[B3] %
    The $\nabla\pi_\ell$ satisfy
    \begin{equation}\label{eq:app1}
      \|\nabla\log\pi_\ell(x) - \nabla\log\pi(x)\|\le \Capp 2^{-\beta\ell},
    \end{equation}
    for all $x\in\R^d$ and all $\ell\in\N_0$.
    \end{enumerate}
  \end{assumption}
  \begin{remark}
    Using the triangle inequality, \eqref{eq:app1} implies
    \begin{equation}\label{eq:app2}
      \|\nabla\log\pi_\ell(x) - \nabla\log\pi_{\ell-1}(x)\|\le 2\Capp 2^{-\beta\ell}
    \end{equation}
    for all $\ell\in\N$.
  \end{remark}
Similar as in \Cref{sec:SVGD},
we introduce again several stochastic dynamical systems,
which will be required for the definition and analysis
of our multilevel scheme.
\begin{enumerate}
\item {\bf Interacting particles:} For each $\ell\in\N_0$
  we let
  $\mathcal X_n^\ell = \{X_n^{\ell,(j)},j=1,\dots,N_\ell\}_{n\ge0}$
  be a particle system of ensemble size $N_\ell\in\N$ evolving through
\begin{equation}\label{eq:SVGD_level} 
\begin{split}
X_{n+1}^{\ell,(i)} &= X_n^{\ell,{(i)}} - \gamma P_{\xell_n}\nabla \log\left(\frac{\xell_n}{\pi_\ell}\right)(X_n^{\ell,{(i)}}) \\ &=: X_n^{\ell,{(i)}}-\gamma \widehat R_n^\ell(X_n^{\ell,{(i)}})
\end{split}
\end{equation}
with i.i.d.~initialization %
$X_0^{\ell,{(i)}}\sim\mfsingle_0$,
$i=1,\dots,N_\ell$, $\ell\in\N_0$.
The corresponding %
empirical measure is denoted by
$\xell_n :=
\frac{1}{N_\ell}\sum_{j=1}^{N_\ell}\delta_{X_n^{\ell,(j)}}$.

\item {\bf Mean field:} For each $\ell\in\N_0$,
  the corresponding MF system is given by
 \begin{equation}\label{eq:MF_level}
 \begin{split}
 Z_{n+1}^{\ell} &= Z_n^{\ell} - \gamma P_{\mfell_n}\nabla\log\left(\frac{\mfell_n}{\pi_\ell}\right)(Z_n^{\ell}) \\ &= Z_n^\ell - \gamma R_n^\ell(Z_n^\ell),\quad R_n^\ell(\cdot) := R_{\mfell_n}(\cdot),
 \end{split}
 \end{equation}
 where $\mfell_n$ denotes the law of the random variable $Z_n^\ell$, i.e.~$(\mfell_n)_{n\ge0}$ solves the MF equation
 \begin{equation}
    \mfell_{n+1} = \left(I-\gamma P_{\mfell_n}\nabla\log\left(\frac{\mfell_n}{\pi_\ell}\right)\right)_{\sharp} \mfell_n,
  \end{equation}
  with i.i.d.~initialization $Z_0^{\ell}\sim \rho_0$, $\ell\in\N_0$.
\item {\bf Mirrored mean field:} For each $\ell\in\N_0$,
  we consider $\mathcal Z_n^{\ell} = \{Z_n^{\ell,(j)},\ j=1,\dots,N_\ell\}$, $n\in\N$ with empirical measure $\mirrormfell_n = \frac{1}{N_\ell}\sum_{j=1}^N \delta_{Z_n^{\ell,(j)}}$.
  It mirrors the MF system in the sense
    \begin{equation}\label{eq:MF_high}
      Z_{n+1}^{\ell,(i)} = Z_{n}^{\ell,(i)} - \gamma R_n^\ell(Z_{n}^{\ell,(i)})
    \end{equation}
    with initialization equal to the $X_0^{(i)}$, meaning
    \begin{equation*}
      Z_{0}^{\ell,(i)}(\omega) = X_0^{\ell,(i)}(\omega)
      \qquad
      \text{for all }\omega\in\Omega,~i=1,\dots,N_\ell.
    \end{equation*}
\end{enumerate}

Except for the initialization, the above systems are analogue to the
ones in Section \ref{sec:SVGD}, but using the dynamics driven by
$\pi_\ell$ instead of $\pi$. Additionally, we'll require two more
auxiliary dynamics:

\begin{enumerate}\setcounter{enumi}{3}
\item {\bf Auxiliary interacting particles:}  
  For $\ell\in\N_0$ we let
  $\widetilde{\mathcal X}_n^{\ell}=\{ \tilde X_n^{\ell,(j)},j=1,\dots,N_{\ell+1}\}$,
  be defined through
 \begin{equation}\label{eq:SVGD_level2} 
 \begin{split}
   \widetilde X_{n+1}^{\ell,(i)} &=  
    \widetilde X_n^{\ell,(i)} - \gamma P_{\xelltilde_n}\nabla \log\left(\frac{\xelltilde_n}{\pi_\ell}\right)(\widetilde X_n^{\ell,{(i)}}) \\ &=: \widetilde X_n^{\ell,(i)} - \gamma \widetilde{\widehat R}_n^{\ell}(\widetilde X_n^{\ell,(i)})\quad i=1,\dots,N_{\ell+1}
     \end{split}
 \end{equation}
 with initialization
 \begin{equation}\label{eq:tildeXellinit}
   \widetilde X_0^{\ell,(i)}(\omega) = X_0^{\ell+1,{(i)}}(\omega), \quad i=1,\dots,N_{\ell+1}
 \end{equation}
 for all $\omega\in\Omega$.
 The corresponding empirical measure is denoted by
 $\xelltilde_n = \frac1{N_{\ell+1}}\sum_{j=1}^{N_{\ell+1}}
 \delta_{\widetilde X_n^{\ell,(j)}}$, $\ell=0,\dots,L-1$.

\item {\bf Auxiliary mirrored mean field:}
    For $\ell\in\N_0$ we let
$\widetilde{\mathcal Z}^\ell_n=\{\widetilde Z_n^{\ell,(j)},j=1,\dots,N_{\ell+1}\}_{n\ge0}$ 
be defined as
    \begin{equation}\label{eq:MF_low}    
      \widetilde Z_{n+1}^{\ell,(i)} = \widetilde Z_{n}^{\ell,(i)} - \gamma R_n^\ell(\widetilde Z_{n}^{\ell,(i)})\quad i=1,\dots,N_{\ell+1}
\end{equation}
with initialization
\begin{equation}\label{eq:tildeZellinit}
  \widetilde Z_{0}^{\ell,(i)}(\omega) = X_0^{\ell+1,(i)}(\omega)
\quad i=1,\dots,N_{\ell+1}
    \end{equation}
    for all $\omega\in\Omega$. As before, note that $R_n^{\ell}$ in
 \eqref{eq:MF_level} independent of the states
 $\widetilde Z_n^{\ell,(j)}$, $j=1,\dots,N_{\ell+1}$.    
    We denote the corresponding
    measure by
    $\mirrormfelltilde_n = \frac{1}{N_{\ell+1}}\sum_{j=1}^{N_{\ell+1}}
    \delta_{\widetilde Z_n^{(j)}}$. 
\end{enumerate}

Similar as in Section~\ref{sec:SVGD}, we consider $(\mathcal X_n^\ell)_{n\ge0}$, $(\widetilde{\mathcal X}_n^\ell)_{n\ge0}$, $(\mathcal Z_n^\ell)_{n\ge0}$ and $(\widetilde{\mathcal Z}_n^\ell)_{n\ge0}$ as stochastic processes on the same probability space $(\Omega,\mathcal F,\mathbb P_0)$ under which the initial conditions described above are satisfied for all $\omega\in\Omega$.

\begin{remark}\label{rem:mirror_particles}
  By definition $\mathcal Z_n^{\ell}$ is an i.i.d.~sample of
  $\mfell_n$, and in particular each $Z_n^{\ell,(j)}$,
  $j=1,\dots,N_\ell$, is a random variable with the same distribution
  as $Z_n^\ell$. Similarly, %
  $\widetilde{\mathcal Z_n^\ell}$ is an i.i.d.~sample %
  of $\mfell_n$, however, of size $N_{\ell+1}$. Hence, %
  the joint random variable
    $(Z_n^{\ell,(i)},\widetilde{Z}_n^{\ell-1,(i)})_{i=1,\dots,N_\ell}$ %
    has marginals $Z_n^{\ell,(i)}\sim\mfell_n$ and
  $\widetilde Z_n^{\ell-1,(i)}\sim\mfellminus_n$. It is important to
  note that due to the initial condition these random variables are correlated which will be a
  crucial %
  to obtain a
  variance reduction %
  for our multilevel
  mean-field estimator.
\end{remark}

We again summarize the introduced systems in \Cref{table:2}.

\begin{table}[htb!]
  \begin{center}
  {\footnotesize
\begin{tabular}{|l|c|c|c|c|c|} 
\hline
 \ & Notation & Evolution & Initialization & Distribution & Size \\
  \hline
  IP & $(\mathcal X_n^\ell,\xell_n)$ & \eqref{eq:SVGD_level}  & $X_0^{\ell,(j)}\sim \rho_0$ & $\xell_n = \frac{1}{N_\ell}\sum_{j=1}^{N_\ell} \delta_{X_n^{\ell,(j)}}$  & $N_{\ell}$  \\  
  MF & $(Z_n^\ell,\mfell_n)$ & \eqref{eq:MF_level} & $Z_0^\ell \sim \rho_0$ & $Z_n^\ell\sim\mfell_n$  & 1  \\
  MMF  & $(\mathcal Z_n^\ell,\mirrormfell_n)$ & \eqref{eq:MF_high} & $Z_0^{\ell,(j)} = X_0^{\ell,(j)}$ & $\mirrormfell_n = \frac{1}{N_\ell}\sum_{j=1}^{N_\ell} \delta_{Z_n^{\ell,(j)}}$ & $N_{\ell}$  \\
  Auxiliary IP & $(\widetilde{\mathcal X}_n^\ell,\xelltilde_n)$ & \eqref{eq:SVGD_level2} & $\tilde X_0^{\ell,(j)} = X_0^{\ell+1,(j)}$ & $\xelltilde_n = \frac{1}{N_{\ell+1}}\sum_{j=1}^{N_{\ell+1}} \delta_{\widetilde X_n^{\ell,(j)}}$  & $N_{\ell+1}$ \\  
  Auxiliary MMF  & $(\widetilde{\mathcal Z}_n^\ell,\mirrormfelltilde_n)$ & \eqref{eq:MF_low} & $\tilde Z_0^{\ell,(j)} = X_0^{\ell+1,(j)}$ & $\mirrormfelltilde_n = \frac{1}{N_{\ell+1}}\sum_{j=1}^{N_{\ell+1}} \delta_{\widetilde Z_n^{\ell,(j)}}$ & $N_{\ell+1}$ \\
  \hline
\end{tabular}
\caption{Overview of the of the five particle dynamics introduced in Section \ref{sec:MLSVGD}: mean field (MF), mirrored mean field (MMF), interacting particle (IP) and the auxiliary systems.}\label{table:2}}
\end{center}
\end{table}

We are now in position to introduce our multilevel SVGD scheme:
As common in multilevel algorithms, the idea is to
use a telescoping sum between consecutive levels to reduce the variance.
To this end we
combine particle systems generated with the same initials, but applied on different accuracy levels $\ell$. Given accuracy levels $\ell = 0,\dots,L$ for some $L\ge 1$, we define the following multilevel estimator for $\mfsingle_n[\varphi]$ by
 \begin{equation}\label{eq:ML_estimator}
     \widehat\mfsingle_n^{\ML}[\varphi] := \xellzero_n[\varphi] + \sum_{\ell=1}^L\left(\xell_n[\varphi]-\xelltildeminus_n[\varphi]\right).
 \end{equation}

To analyze it, we work under the following assumption regarding the
computational cost:
 \begin{assumption}\label{ass:cost}
   There exists $q\ge 0$ such that
   the generation of $\mathcal{X}_n^{\ell} = \{X_n^{\ell,(j)},j=1,\dots,N_\ell\}$ defined %
   in \eqref{eq:SVGD_level} %
   has computational cost
    \[\cost(\mathcal{X}_n^\ell) = n\cdot N_\ell \cdot 2^{q\ell}.\]
\end{assumption}

We emphasize that the generation of
$\widetilde{\mathcal{X}}_n^\ell = \{\widetilde X_n^{\ell,(i)},i=1,\dots,N_{\ell+1}\}_{n\ge0}$ in \eqref{eq:SVGD_level2} %
involves
$N_{\ell+1}$ additional evaluations of $\log\pi_\ell$. Thus, %
the computational cost of the estimator $\widehat\mfsingle_n^{\ML}[\varphi]$ is given by
\begin{align*} 
\cost_{\ML} &= n\cdot \left(\sum_{\ell = 0}^L N_\ell \cdot 2^{q\ell} + \sum_{\ell = 1}^L N_\ell \cdot 2^{q(\ell-1)} \right) \\ &\le n\cdot 2^{1-q} \sum_{\ell = 0}^L N_\ell \cdot 2^{q\ell}
\end{align*}
 In Figure~\ref{fig:ML_sketch} in the appendix, we present an illustration of this multilevel particle approximation. In the remaining part of this paper, we will show how this multilevel construction achieves a significant complexity %
reduction compared to the single level estimator.

\section{Convergence analysis}\label{sec:convergence}
\subsection{Main result}\label{ssec:main}
We will prove the following error bound by employing a series of auxiliary results, which have been relocated to Section~\ref{ssec:auxres} for better organization and clarity.
\begin{theorem}[Error bound]\label{thm:err_ML}
    Under Assumption~\ref{ass:kernel2}, \ref{ass:phi}, and \ref{ass:level_accuracy}, we obtain for all $n\le T/\gamma$
    \begin{equation*} 
    \begin{split} 
    &\|\widehat\mfsingle_n^{\ML}[\varphi]- \mfsingle_n[\varphi]\|_{L_0^2(\R)} = \E_0[|\widehat\mfsingle_n^{\ML}[\varphi]- \mfsingle_n[\varphi]|^2]^{\frac12}\\ & \lesssim (L+1)\left(\frac{1}{\sqrt{N_0}} + \sum_{\ell=1}^L \frac{2^{-\beta\ell}}{\sqrt{N_\ell}}\right) + 2^{-\beta L}, 
    \end{split}
    \end{equation*}
    where the constants depend on $\varphi$, $T$ and the constants
      in Assumptions~\ref{ass:kernel2} and \ref{ass:level_accuracy},
    but are independent of $N_\ell$, $\ell=0,\dots L$ and $L>0$. 
  \end{theorem}

  Next we present a ``single-level'' result. Its proof follows
  by similar arguments as Proposition~\ref{prop:err_MC}.
  \begin{theorem}
    Under Assumption~\ref{ass:kernel2} and \ref{ass:level_accuracy}, we obtain for all $n\le T/\gamma$
    \begin{equation*}
    \begin{split}
    \|\xellL_n[\varphi] - \mfsingle_n[\varphi]\|_{L_0^2(\R)}&=\E_0[|\xellL_n[\varphi] - \mfsingle_n[\varphi]|^2]^{\frac12}\\ &\lesssim \frac{1}{\sqrt{N_L}}+ 2^{-\beta L}\,
    \end{split}
    \end{equation*}
    where the constants depend on $\varphi$, $T$ and the constants
      in Assumptions~\ref{ass:kernel2} and \ref{ass:level_accuracy},
    but are independent of $N_L\in\N$ and $L>0$. 
  \end{theorem}
We will quantify both the single-level and the multilevel complexity in Theorem~\ref{thm:complexity_SL} and Theorem~\ref{thm:complexity_ML}.

\subsection{Stability of the mean-field equation}\label{ssec:stability}
We next provide a %
stability result for the MF equation \eqref{eq:MF_level}.
Apart from being required in our convergence analysis,
stability
is important from an inverse problems point of view. %
Inverse problems are typically ill-posed and real-world measurements
are often subject to noise or uncertainties. Therefore, it is crucial
to understand the sensitivity of SVGD 
w.r.t.\ changes in $\log \pi_\ell$.

\begin{proposition}\label{prop:stability}
    Under Assumption~\ref{ass:kernel2} and \ref{ass:level_accuracy} the MF limit \eqref{eq:MF_level} is stable w.r.t.~changes in $\pi_\ell$, in the sense that
    \begin{align*} 
    \mathcal W_2(\mfell_n,\mfellminus_n)&\le \mathbb E_0[\|Z_n^\ell-Z_n^{\ell-1}\|^2]^{\frac12}\\ 
    &\le%
    \Capp 2^{-\beta \ell} (e^{2\dlip T}-1),
    \end{align*}
    and
    \begin{align*} 
    \mathcal W_2(\mfell_n,\mfsingle_n)&\le \mathbb E_0[\|Z_n^\ell-Z_n\|^2]^{\frac12} \\ &\le\frac{\Capp 2^{-\beta \ell}}{2}(e^{2\dlip T}-1),
    \end{align*}
    for all $n\le T/\gamma$, where $\dlip>0$ is defined in Lemma~\ref{lemma:dlip}.
\end{proposition}

\subsection{Variance reduction: Combined stability and mean-field limit}\label{ssec:auxres}
In the following, we present the main advantage of our proposed multilevel MF approach. Through incorporation of the telescoping sum into the estimator \eqref{eq:ML_estimator}, we obtain a variance reduction for the differences 
$\xell_n[\varphi]-\mirrormfell_n[\varphi]$ for increasing accuracy $\ell$. In order to achieve the variance reduction, we combine the MF limit in Proposition~\ref{prop:MF} and the stability result in Proposition~\ref{prop:stability}.

\begin{lemma}\label{lem:MF_levels}
    Under Assumption~\ref{ass:kernel2} and \ref{ass:level_accuracy}, we have for all $n\le T/\gamma$ that
     \[\E_0[\|X_n^{\ell,(1)}-Z_n^{\ell,(1)}\|^2]^{\frac12}\lesssim \frac{1}{\sqrt{N_0}}+\sum_{m=1}^\ell \frac{2^{-\beta m}}{\sqrt{N_m}},\quad \ell = 0,\dots, L.\]
    The constants depend on $T,\ \gamma,\ \clip>0$ and the constants
      in Assumptions~\ref{ass:kernel2} and \ref{ass:level_accuracy}, but are independent of $N_\ell$, $\ell=0,\dots L$ and $L>0$.
\end{lemma}

\section{Complexity analysis}\label{sec:complexity}
Having derived the improved error bound for the proposed multilevel MF approximation \eqref{eq:ML_estimator}, we now want to compare both the multilevel and single-level MF approximation. We refer to $\xellL_n[\varphi]$ as single-level MF approximation, which applies the particle approximation of SVGD to one fixed accuracy level $L$ and a fixed number of particles $N_L$. 
\begin{theorem}[Single-level complexity]\label{thm:complexity_SL}
    Suppose that Assumptions~\ref{ass:kernel2}, \ref{ass:level_accuracy} and \ref{ass:cost} are satisfied. Moreover, given a tolerance $\varepsilon>0$ let $L= \ceil{\frac{1}{\beta\log(2)}\log(\frac2\varepsilon)}$ and $N_L \propto \varepsilon^{-2}$. Then the expected error of the SL estimator is bounded by
    \[\|\xellL_n[\varphi]- \mfsingle_n[\varphi]\|_{L_0^2(\R)} = \E_0[|\xellL_n[\varphi]- \mfsingle_n[\varphi]|^2]^{\frac12}\lesssim \varepsilon\]
    with a cost that is bounded by 
    \[\cost_{\SL} = n \cdot N_L \cdot 2^{qL} \lesssim \varepsilon^{-2-\frac{q}{\beta}}.\]
  \end{theorem}

Depending on the relation between computational cost parameter $q\ge0$ and approximation parameter $\beta>0$, we are able to verify improved rates of convergence for the proposed multilevel MF estimator \eqref{eq:ML_estimator}.
\begin{theorem}[Multilevel complexity]\label{thm:complexity_ML}
    Suppose that Assumptions~\ref{ass:kernel2}, \ref{ass:level_accuracy} and \ref{ass:cost} are satisfied. Moreover, given a tolerance $\varepsilon>0$ let $L= \ceil{\frac{1}{\beta\log(2)}\log(\frac2\varepsilon)}$ and $N_\ell \propto (L+1)^4 2^{-2\beta\ell}\varepsilon^{-2}$. Then the expected error of the ML estimator is bounded by
    \[\|\widehat\mfsingle_n^{\ML}[\varphi]- \mfsingle_n[\varphi]\|_{L_0^2(\R)} = \E_0[|\widehat\mfsingle_n^{\ML}[\varphi]- \mfsingle_n[\varphi]|^2]^{\frac12}\lesssim \varepsilon\]
    with a cost that is bounded by
    \begin{align*} 
    \cost_{\ML} &= n\cdot \left(\sum_{\ell = 0}^L N_\ell \cdot 2^{q\ell} + \sum_{\ell = 1}^L N_\ell \cdot 2^{q(\ell-1)} \right)
    \\ &\lesssim \begin{cases}
        |\log(\varepsilon)|^4\varepsilon^{-\frac{q}{\beta}}, & q>2\beta,\\
        |\log(\varepsilon)|^5\varepsilon^{-2}, &q=2\beta,\\
        |\log(\varepsilon)|^4\varepsilon^{-2}, &q<2\beta.
    \end{cases} 
    \end{align*}
\end{theorem}

\section{Numerical results}
Let $D=(0,1)$ and consider the inverse problem of recovering $f\in L^2(D)$ given discrete observation points of the solution $u_f\in H^2(D) \cap H_0^1(D) \subset L^2(D)$ of the equation 
\begin{equation}\label{eq:IP_example}
    \begin{split}
      -u_f''(s)
        + u_f(s) &= f(s),\quad s\in D,\\
        u_f(s) &=0,\quad s\in\partial D\,.
    \end{split}
\end{equation}
We define the solution-to-observation operator as bounded linear mapping $\mathcal O:%
H_0^1(D)
\to \R^{n_y}$ such that the forward model is given by $f\mapsto F(f):=\cO(u_f)\in\R^{n_y}$. In our specific example we define $\cO(u_f):= (u_f(s_i))_{i=1}^{n_y}$ with $s_i=\frac{i}{n_y+1}$ and $n_y=15$. Since \eqref{eq:IP_example} has no closed form solution, $u_f$ needs to be approximated numerically using e.g.~the finite element method (FEM). More precisely, given an accuracy level $\ell\ge1$ we consider piecewise linear FEM on a uniform mesh over $D$ of size $2^\ell$ to obtain the approximation $u_f^\ell$ satisfying $\|u_f-u_f^\ell\|_{H_0^1(D)}\lesssim 2^{-\ell}$ and $\|u_f^{\ell}-u_f^{\ell-1}\|_{L^2(D)}\lesssim 2^{-\ell}$.

We introduce a parametrization of $f\in L^2(D)$ through the (truncated) spectral representation 
\[ f(x,\cdot) = \sum_{i=1}^d x_i \frac{\sqrt{2}}{\pi} \sin(i\pi \cdot)\in L^2(D)\]
for $x = (x_1,\dots,x_d)^\top\in\R^d$. Our prior on the parameters $x$ is given as Gaussian $\cN(0,C_0)$ with $C_0 = \diag(i^{-2},i=1,\dots,d)$. Using the FEM approximation $u_{f(x,\cdot)}^\ell$ we can construct an approximation $\log \pi_\ell(x)$ of $\log\pi(x)$ satisfying 
\[\| \nabla \log\pi_\ell(x)-\nabla\log\pi(x)\| \lesssim 2^{-\ell},\quad \ell\in\N\,. \]
For further details on the approximation error we refer to \cite[Section~6]{MLOPTI}, where a similar example was discussed.

For the implementation of the discrete-time SVGD in the different
considered variants \eqref{eq:SVGD_level} and \eqref{eq:SVGD_level2},
we fix a Gaussian kernel
$k(x_1,x_2) \propto \exp(-\frac{1}{2}\|C_0^{-1/2}(x_1-x_2)\|^2)$ and
set a step size $\gamma = 10^{-1}$. To generate a reference solution
of the MF equation, we have applied \eqref{eq:SVGD_level} with
accuracy level $\ell = L_{\mathrm{ref}} = 13$ and ensemble size
$N_{\mathrm{ref}} = 3000$, which we denote by
$\rho_n^{\mathrm{ref}}[\varphi]$. Moreover, we have followed the
choices on $L$, $N_L$ and $L$, $N_\ell$, $\ell = \ell_0,\dots, L$
proposed in Theorem~\ref{thm:complexity_SL} and
Theorem~\ref{thm:complexity_ML} to construct our single-level
estimator $\xellL_n[\varphi]$ and multilevel estimator
$\widehat\mfsingle_n^{\ML}[\varphi]$. As quantity of interest, we have
considered $\varphi(x) = \|f(x,\cdot)\|_{L^2(D)}$ which for each given
$x\in\R^d$ is approximated on the finest grid $L_{\mathrm{ref}}$ using
a trapezoid quadrature rule. We have applied $100$ runs to construct
Monte Carlo estimates of the error
$\E[\|\widehat\mfsingle_n^{\ML}[\varphi] -
\rho_n^{\mathrm{ref}}[\varphi]\|^2]^{\frac12}$ and
$\E[\|\xellL_n[\varphi] -
\rho_n^{\mathrm{ref}}[\varphi]\|^2]^{\frac12}$ which are shown in
Figure~\ref{fig:errorplot_SVGD} for different choices of $n$.

\begin{figure*}[htb!]
\centering \includegraphics[width=0.33\textwidth]{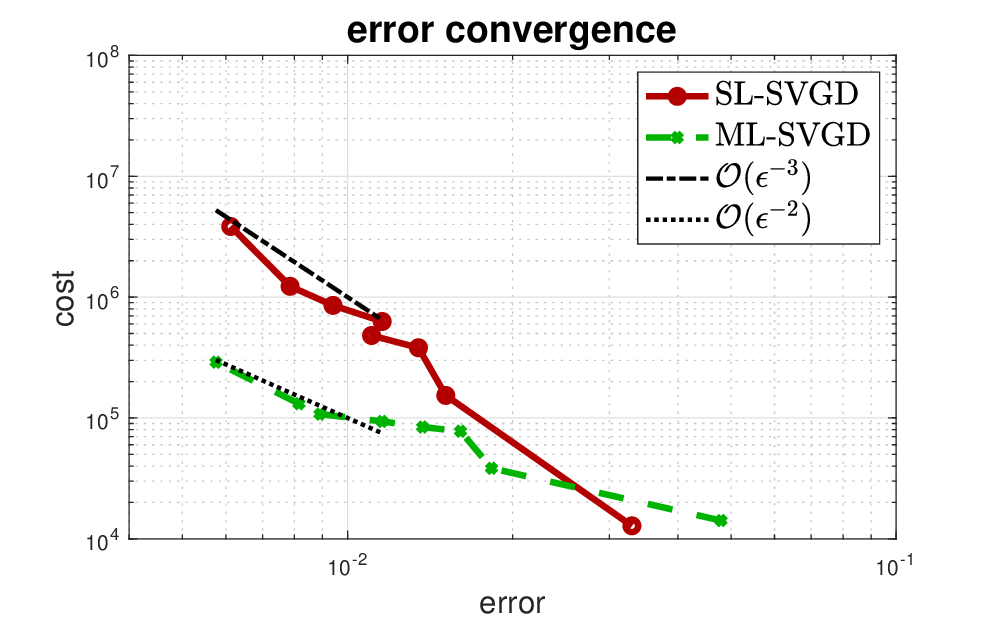}~~\includegraphics[width=0.33\textwidth]{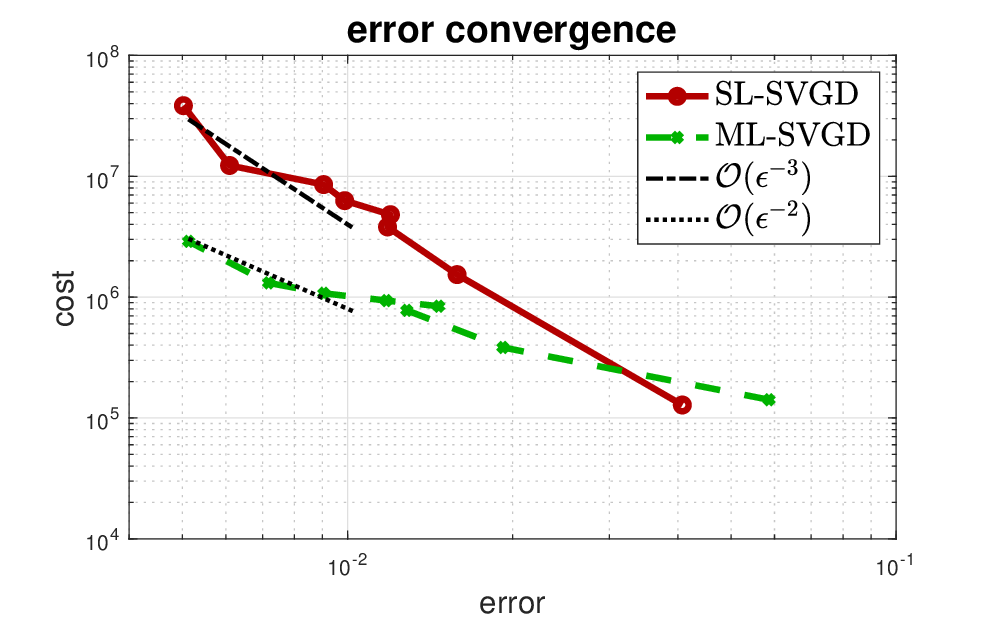}~~\includegraphics[width=0.33\textwidth]{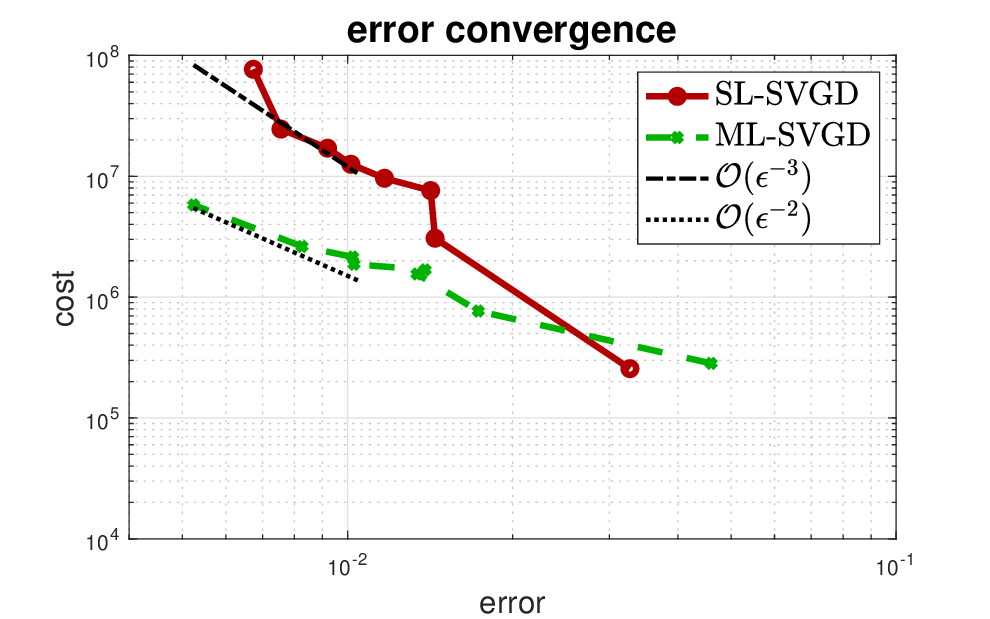}
 \caption{Error convergence for $10$ iterations (left), $100$ iterations (middle) and $200$ iterations (right).} \label{fig:errorplot_SVGD}
\end{figure*} 

\section{Conclusions}
In this paper, we introduced a novel SVGD multilevel method and
  provided a convergence analysis.  As an application we focus on
  Bayesian inverse problems, for which the likelihood is expensive to
  evaluate, but numerical approximations at different accuracy levels
  and computational complexity are available. Our main results give
  error bounds in terms of the total computational complexity of all
  required likelihood evaluations; in particular we show an
  improvement over a ``naive'' single-level implementation of SVGD.

  Our method, which is based on a meticulous combination of several
  particle systems operating at different accuracy levels,
  fundamentally differs from previous multilevel SVGD and interacting
  particle approaches, such as those in \cite{alsup22a,MLOPTI}. In
  these works, the authors propose incrementally increasing the
  accuracy of likelihood evaluations over time steps. Combining this
  idea with our approach presents an interesting opportunity for
  future research.

\bibliographystyle{abbrv}
\bibliography{references.bib}

\newpage
\appendix
\section{Illustration of the multilevel mean-field particle approximation}
\begin{figure*}[htb!]
\centering \includegraphics[width=0.85\textwidth]{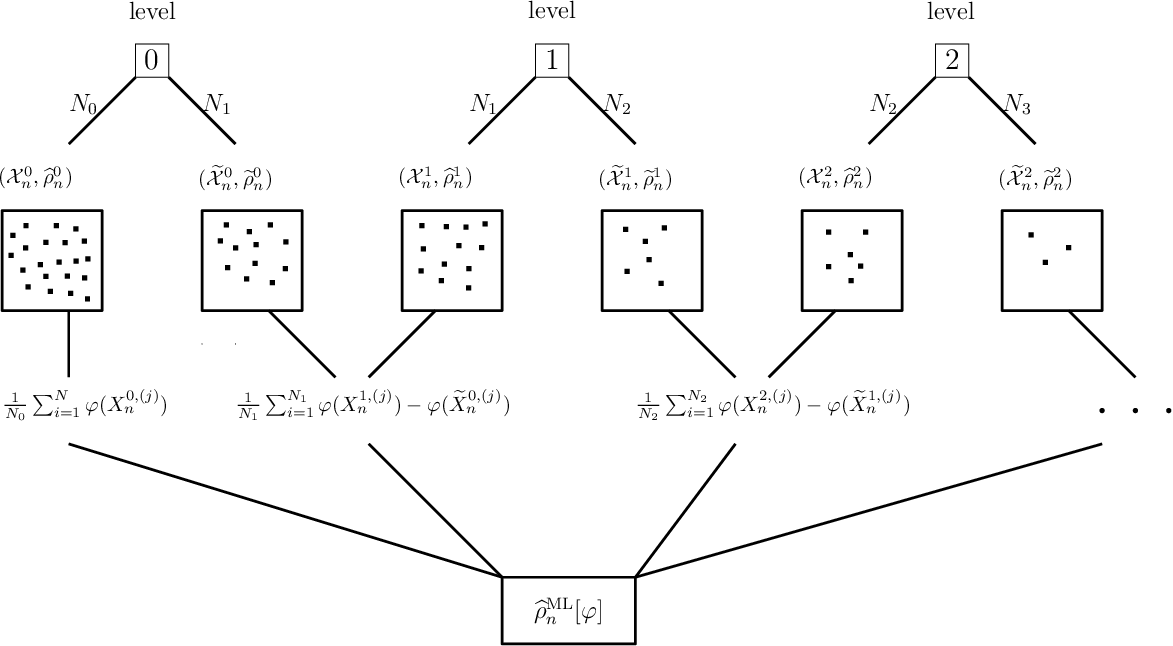}
 \caption{Illustration of the multilevel mean-field particle approximation.} \label{fig:ML_sketch}
\end{figure*} 

\section{Proofs of \Cref{sec:MLSVGD} }
\begin{proof}[Proof of \Cref{prop:err_MC}]
    We can quantify the approximation error through
    \[\E_0[|\xsingle_n[\varphi] - \mfsingle_n[\varphi]|^2]^{\frac12} \le \E_0[|\xsingle_n[\varphi]- \mirrormfsingle_n[\varphi]|^2]^{\frac12} + \E_0[|\mirrormfsingle_n[\varphi] - \mfsingle_n[\varphi]|^2]^{\frac12} =: A_1 + A_2.\]
    We start with
    \begin{align*} 
    A_1 = \E_0[|\frac{1}N\sum_{i=1}^N \left(\varphi(X_n^{(i)})-\varphi(Z_n^{(i)})\right)|^2]^{\frac12} &\le \frac{\alip}N\sum_{i=1}^N \E_0[\|X_n^{(i)} - Z_n^{(i)}\|^2]^{\frac12} \\ &\le \alip \left(\frac1N\sum_{i=1}^N \E_0[\|X_n^{(i)} - Z_n^{(i)}\|^2]\right)^{1/2}=\alip c_n,
    \end{align*}
    where we have used Jensen's inequality. Note that $c_n\le\frac12\left(\frac{1}{\sqrt{N}} \sqrt{\Var(\mfsingle_0)} e^{A T}\right)(e^{2 A T}-1)$ by Proposition~\ref{prop:MF}. %
    Moreover, we have
    \begin{align*}
        A_2 = \E_0[|\frac{1}N\sum_{i=1}^N \left(\varphi(Z_n^{(i)})-\mathbb E_0[\varphi(Z_n)]\right)|^2]^{\frac12} \le B\frac{\E_0[|\varphi(Z_n)-\E_0[|\varphi(Z_n)|]|^2]}{\sqrt{N}},
    \end{align*}
    where we have used the Marcinkiewicz-Zygmund
    inequality, see \cite[Theorem 5.2]{KM2015} and the fact that the $Z_n^{(i)}$ are i.i.d.~%
    with the same distribution as $Z_n$.
    Note that due to Lipschitz continuity of $\varphi$ we have \[\E_0[|\varphi(Z_n)-\E_0[\varphi(Z_n)]|^2]\le\int_{\R^d} |\varphi(z)|^2\,\mfsingle_n({\mathrm d}z)\le \varphi^2(0) + 2\alip\int_{\R^d} |z|^2\,\mfsingle_n({\mathrm d}z) \le 2\varphi^2(0) + \alip \sqrt{\Var(\mfsingle_n)}\] and %
    according to
    \cite[Lemma~12]{NEURIPS20203202111c} under Assumption~\ref{ass:kernel2} the variance $\Var(\mfsingle_n)$ remains bounded for $n\le T/\gamma$ due to \eqref{eq:initial_moment}. Finally, we obtain
    $A_1+A_2 \le C/\sqrt{N}$ for some constant $C>0$ independent of $N$. 
\end{proof}

\section{Proofs of \Cref{sec:convergence}}
\subsection{Proofs of \Cref{ssec:main}}
\begin{proof}[Proof of \Cref{thm:err_ML}]
    The principle of the derived error bounds follows the multilevel ensemble Kalman filtering formulation presented in \cite{doi:10.1137/15M100955X}.
    We define the estimators
    \[\mirrormfell_n[\varphi] = \frac1{N_\ell} \sum_{i=1}^{N_\ell}\varphi(Z_n^{\ell,(i)}) \quad \text{and}\quad \mirrormfell_n[\varphi]-\mirrormfelltildeminus_n[\varphi] = \frac1{N_\ell} \sum_{i=1}^{N_\ell}\left(\varphi(Z_n^{\ell,(i)})-\varphi(\widetilde Z_n^{\ell-1,(i)})\right),\]
    and the corresponding multilevel estimator
    \[\mirrormfsingle_n^{\ML}[\varphi] = \mirrormfellzero_n + \sum_{\ell=1}^L \left(\mirrormfell_n[\varphi]-\mirrormfelltildeminus_n[\varphi]\right).\]
    We can split the error into
    \begin{equation}\label{eq:splitted_MLerr}
    \E_0[|\widehat\mfsingle_n^{\ML}[\varphi]- \mfsingle_n[\varphi]|^2]^{\frac12}\le \E_0[|\widehat\mfsingle_n^{\ML}[\varphi]- \mirrormfsingle_n^{\ML}[\varphi]|^2]^{\frac12} + \E_0[|\mirrormfsingle_n^{\ML}[\varphi]- \mfellL_n[\varphi]|^2]^{\frac12} + |\mfellL_n[\varphi]-\mfsingle_n[\varphi]|.
    \end{equation}
    We start with the first term on the rhs in \eqref{eq:splitted_MLerr}.
    Using that $\tilde X_{0}^{\ell,(i)}= X_0^{\ell+1,(i)}=\tilde Z_0^{\ell,(i)}$ 
      for $i=1,\dots,N_{\ell+1}$ (cp.~\eqref{eq:tildeXellinit} and \eqref{eq:tildeZellinit}),
    \begin{align*}
    &\E_0[|\widehat\mfsingle_n^{\ML}[\varphi]- \mirrormfsingle_n^{\ML}[\varphi]|^2]^{\frac12}\\ &\le \frac{1}{N_0}\sum_{i=1}^{N_0}\E_0[|\varphi(X_n^{0,(i)})-\varphi(Z_n^{0,(i)})|^2]^{\frac12} \\ &\quad + \sum_{\ell=1}^L \frac{1}{N_\ell}\sum_{i=1}^{N_\ell} \E_0[|\varphi(X_n^{\ell,(i)})-\varphi(\tilde X_n^{\ell-1,(i)})-(\varphi(Z_n^{0,(i)})-\varphi(\tilde Z_n^{\ell-1,(i)}))|^2]^{\frac12}\\
    &= \E_0[|\varphi(X_n^{0,(1)})-\varphi(Z_n^{0,(1)})|^2]^{\frac12} + \sum_{\ell=1}^L \E_0[|\varphi(X_n^{\ell,(1)})-\varphi(\tilde X_n^{\ell-1,(1)})-(\varphi(Z_n^{0,(1)})-\varphi(\tilde Z_n^{\ell-1,(1)}))|^2]^{\frac12}\\
    &\le \alip \E_0[\|X_n^{0,(1)}-Z_n^{0,(1)}\|^2]^{\frac12} + \alip\sum_{\ell=1}^L \left(\E_0[\|X_n^{\ell,(1)}-Z_n^{\ell,(1)}\|^2]^{\frac12}+\E_0[\|\widetilde X_n^{\ell-1,(1)}-\widetilde Z_n^{\ell-1,(1)}\|^2]^{\frac12}\right).
    \end{align*}  
    Here we used that $\{X_n^{\ell,(j)}, j=1,\dots,N_\ell\}$ and
    $\{Z_n^{\ell,(j)}, j=1,\dots,N_\ell\}$ are %
    collections of
    identically distributed (but not independent) random variables.
    Applying Lemma \ref{lem:MF_levels} this yields
    \begin{equation*}
      \E_0[|\widehat\mfsingle_n^{\ML}[\varphi]- \mirrormfsingle_n^{\ML}[\varphi]|^2]^{\frac12}
      \lesssim \sum_{\ell=0}^L \Big(\frac{1}{\sqrt{N_0}}+\sum_{m=1}^\ell \frac{2^{-\beta m}}{\sqrt{N_m}}\Big)\le
      (L+1) \left(\frac{1}{\sqrt{N_0}}+ \sum_{\ell=1}^L \frac{2^{-\beta \ell}}{\sqrt{N_\ell}} \right).
    \end{equation*}
    Next, for the second term of the rhs in \eqref{eq:splitted_MLerr} we again apply
    the Marcinkiewicz-Zygmund inequality \cite[Theorem~5.2]{KM2015},
    \begin{align*}
    &\E_0[|\mirrormfsingle_n^{\ML}[\varphi]-\mfellL_n[\varphi]|^2]^{\frac12}= \E_0\Big[\Big|\mirrormfellzero_n[\varphi]-\mfellzero_n[\varphi] + \sum_{\ell=1}^L\left(\mirrormfell_n[\varphi] - \mirrormfelltildeminus_n[\varphi] - (\mfell_n[\varphi]-\mfellminus_n[\varphi])\right)\Big|^2\Big]^{\frac12}\\
    &\le \E_0\Big[\Big|\frac{1}{N_0}\sum_{i=1}^{N_0} \varphi(Z_n^{0,(i)})-\E_0[\varphi(Z_n^{0})]\Big|^2\Big]^{\frac12}\\ &\quad + \sum_{\ell=1}^L \E_0\Big[\Big|\frac{1}{N_\ell}\sum_{i=1}^{N_\ell}\left(\varphi(Z_n^{\ell,(i)})-\varphi(\widetilde Z_n^{\ell-1,(i)}) - \E_0[\varphi(Z_n^{\ell})-\varphi( Z_n^{\ell-1})]\right)\Big|^2\Big]^{\frac12}\\
    &\le \frac1{\sqrt{N_0}}\E_0[|\varphi(Z_n^{0})|^2]^{\frac12} + \sum_{\ell=1}^L \frac{1}{\sqrt{N_\ell}} \E_0[|\varphi(Z_n^{\ell})-\varphi(Z_n^{\ell-1})|^2]^{\frac12}\\
    &\le \frac{\varphi(0)+\alip}{\sqrt{N_0}}\E_0[\|Z_n^{0}\|^2]^{\frac12}+ \sum_{\ell=1}^L \frac{\alip}{\sqrt{N_\ell}} \E_0[\|Z_n^{\ell}-Z_n^{\ell-1}\|^2]^{\frac12},\\
      &\le \frac{\varphi(0)+\alip}{\sqrt{N_0}}\E_0[\|Z_n^{0}\|^2]^{\frac12} + \sum_{\ell=1}^L \frac{\alip}{\sqrt{N_\ell}} \frac{\Capp 2^{-\beta \ell}}{2}(e^{2\gamma\dlip T}-1)\\
      &\lesssim \frac{1}{\sqrt{N_0}}+\sum_{\ell=1}^L \frac{2^{-\beta\ell}}{\sqrt{N_\ell}}
    \end{align*}
    where for the second inequality we have used
    Remark~\ref{rem:mirror_particles}, for the second to last
    inequality we used Proposition~\ref{prop:stability}, and $\dlip>0$
    is as in Lemma \ref{lemma:dlip}.
    
    Finally, for the third term on the rhs in
    \eqref{eq:splitted_MLerr} again with
    Proposition~\ref{prop:stability} we obtain
    \[|\mfellL_n[\varphi]-\mfsingle_n[\varphi]| = |\E_0[\varphi(Z_n^L)-\varphi(Z_n)]| \le \alip \E_0[\|Z_n^L-Z_n\|^2]^{\frac12} %
      \lesssim
      2^{-\beta L}.\]
\end{proof}

\subsection{Proofs of \Cref{ssec:stability}}
We first show the following stability property w.r.t.~$\log\pi_\ell$, 
which is similar to Lemma~\ref{lem:Lipschitz_drift}.
\begin{lemma}\label{lemma:dlip}
  Under Assumption~\ref{ass:kernel2} and \ref{ass:level_accuracy}
  there exists $\dlip>0$ depending on the constants in these
  assumptions such that
    \begin{align*}
        &\|P_{\rho_1}\nabla\log\left(\frac{\rho_1}{\pi_\ell}\right)(z_1) - P_{\rho_2}\nabla\log\left(\frac{\rho_2}{\pi_{\ell-1}}\right)(z_2)\|\\ &\le \dlip [\|z_1-z_2\|+ \mathcal W_2(\rho_1,\rho_2)+ 2^{-\beta \ell}].
    \end{align*} 
    and similarly
    \begin{align*} 
    &\|P_{\rho_1}\nabla\log\left(\frac{\rho_1}{\pi_\ell}\right)(z_1) - P_{\rho_2}\nabla\log\left(\frac{\rho_2}{\pi}\right)(z_2)\| \\ &\le \dlip [\|z_1-z_2\|+ \mathcal W_2(\rho_1,\rho_2)+ 2^{-\beta \ell}].
    \end{align*}
  \end{lemma}
\begin{proof}%
    Using the triangle inequality
    \begin{align*} 
    \|P_{\rho_1}\nabla\log\left(\frac{\rho_1}{\pi_\ell}\right)(z_1) - P_{\rho_2}\nabla\log\left(\frac{\rho_2}{\pi_{{\ell-1}}}\right)(z_2)\|&\le  \|P_{\rho_1}\nabla\log\left(\frac{\rho_1}{\pi_\ell}\right)(z_1) - P_{\rho_1}\nabla\log\left(\frac{\rho_1}{\pi_{\ell-1}}\right)(z_1)\|\\ &\quad+ \|P_{\rho_1}\nabla\log\left(\frac{\rho_1}{\pi_{\ell-1}}\right)(z_1)- P_{\rho_2}\nabla\log\left(\frac{\rho_2}{\pi_{\ell-1}}\right)(z_2)\| \\
    &=: I_1 + I_2.
    \end{align*}
    By
    Lemma~\ref{lem:Lipschitz_drift}
    \[I_2 \le \clip[\|z_1-z_2\| + \mathcal W_2(\rho_1,\rho_2)].\]
    For $I_1$ we get
    \begin{align*}
        I_1 &= \|\mathbb E_{X\sim\rho_1}[\nabla\log\pi_\ell(X)k(X,z_1)-\nabla\log\pi_{\ell-1}(X)k(X,z_1) + \nabla_1k(X,z_1)-\nabla_1 k(X,z_1)]\|\\
        &= \|\mathbb E_{X\sim\rho_1}[(\nabla\log\pi_\ell(X)-\nabla\log\pi_{\ell-1}(X))k(X,z_1)]\|\\
        &\le B\mathbb E_{X\sim\rho_1}[\|\nabla\log\pi_{\ell}(X)-\nabla\log\pi_{\ell-1}(X)\|]\le 2 B \Capp 2^{-\beta \ell},
    \end{align*}
    where we have used the boundedness of the kernel and \eqref{eq:app2}. The second assertion follows by similar argumentation.
\end{proof}

\begin{proof}[Proof of \Cref{prop:stability}]
    We consider the two stochastic processes 
    \begin{align*}
        Z_{n+1}^{\ell} &= Z_{n}^{\ell} - \gamma R_n^\ell(Z_{n}^{\ell}),\quad Z_{0}^{\ell} \sim\mfsingle_0, \\
        Z_{n+1}^{\ell-1} &= Z_{n}^{\ell-1} - \gamma R_n^{\ell-1}(Z_{n}^{\ell-1}),
    \end{align*}
    with $\quad Z_{0}^{\ell-1}(\omega) = Z_{0}^{\ell}(\omega)$ for $\mathbb P_0$-almost all $\omega\in\Omega$, such that the Wasserstein-2 distance is lower bounded by
    \[\mathcal W_2(\mfell_n,\mfellminus_n) \le \mathbb E_0[\|Z_n^\ell-Z_n^{\ell-1}\|^2]^{\frac12}.\]
    The quantity $c_n = \mathbb E_0[\|Z_n^\ell-Z_n^{\ell-1}\|^2]^{\frac12}$ evolves in time through
    \begin{align*}
        c_{n+1} &= \E_0[\|Z_n^\ell - Z_n^{\ell-1} -\gamma(R_n^\ell(Z_n^\ell)-R_n^{\ell-1}(Z_n^{\ell-1}))\|^2]^{\frac12}\\
                &\le \E_0[\|Z_n^\ell - Z_n^{\ell-1}\|^2]^{\frac12} + \gamma \E_0[\|R_n^\ell(Z_n^\ell)-R_n^{\ell-1}(Z_n^{\ell-1})\|^2]^{\frac12}\\
                &= c_n + \gamma \E_0[ \|P_{\mfell_n}\nabla\log\left(\frac{\mfell_n}{\pi_\ell}\right)(Z_n^\ell)-P_{\mfellminus_n}\nabla\log\left(\frac{\mfellminus_n}{\pi_{\ell-1}}\right)(Z_n^{\ell-1})\|^2]^{\frac12}\\
                &\le c_n + \gamma \dlip \left( \E_0[\|Z_n^\ell-Z_n^{\ell-1}\|^2]^{\frac12} + \mathcal W_2(\mfell_n,\mfellminus_n)+ 2^{-\beta \ell}\right)\\
                &\le (1+2\gamma \dlip) c_n + \gamma\dlip 2^{-\beta\ell},
    \end{align*}
    where we used Lemma \ref{lemma:dlip}.
    Since $c_0=0$, we conclude with discrete Gronwall inequality, Lemma~\ref{lem:discr_Gr}, that
    \[c_n \le \frac{2^{-\beta \ell}}{2}(e^{2\dlip T}-1)\, .\]
    The second assertion follows again by similar argumentation.
\end{proof}

Similarly, one can derive the following stability result on a particle level when applied to our mirrored mean field and auxiliary mirrored mean field particle systems.

\begin{proposition}\label{prop:stability_particles}
    Under Assumption~\ref{ass:kernel2} and \ref{ass:level_accuracy} the mirrored MF limit is stable w.r.t.~changes in $\pi_\ell$, in the sense that
    \begin{align*} 
    \mathcal W_2(\mirrormfell_n,\mirrormfelltildeminus_n)&\le \mathbb E_0[\|Z_n^{\ell,(1)}-\widetilde Z_n^{\ell-1,(1)}\|^2]^{\frac12}\\ 
    &\le%
    \Capp 2^{-\beta \ell} (e^{2\dlip T}-1),
    \end{align*}
    for all $n\le T/\gamma$, where $\mirrormfell_n$ and $\mirrormfelltildeminus_n$ are defined in \cref{sec:MLSVGD}. 
\end{proposition}

\subsection{Proofs of \Cref{ssec:auxres}}
\begin{proof}[Proof of \Cref{lem:MF_levels}]
  We are going to prove the claim via induction. On level $\ell=0$, we have for all $n\le T/\gamma$
  \begin{equation}\label{eq:induction_start}
    \E_0[\|X_n^{0,(1)}-Z_n^{0,(1)}\|^2]^{\frac12}\lesssim \frac{1}{\sqrt{N_0}},
  \end{equation}
    due to the MF limit Proposition~\ref{prop:err_MC} (or \cite[Proposition~7]{NEURIPS20203202111c}).
    Now, assume that for all $n\le T/\gamma$ 
    \begin{equation}\label{eq:ind_req} 
   \E_0[\|X_n^{\ell-1,(1)}-Z_n^{\ell-1,(1)}\|^2]^{\frac12}\lesssim \frac{1}{\sqrt{N_0}} + \sum_{m=1}^{\ell-1} \frac{2^{-\beta m}}{\sqrt{N_m}}\,
    \end{equation}
    for some $\ell\ge1$. We define 
    $\Delta_n^{\ell} = \E_0[\|X_{n}^{\ell,(1)} - \widetilde X_{n}^{\ell-1,(1)} - (Z_n^{\ell,(1)}-\widetilde Z_n^{\ell-1,(1)})\|^2]^{\frac12}$ and observe
    the following nested behavior
    \begin{equation} \label{eq:nested}
    \begin{split}
    \E_0[\|X_n^{\ell,(1)} - Z_n^{\ell,(1)}\|^2]^{\frac12} &\le \E_0[\|\widetilde X_n^{\ell-1,(1)}-\widetilde Z_n^{\ell-1,(1)}\|^2] + \Delta_n^\ell\\
    &\le \Delta_n^{\ell-1}+ \E_0[\|X_n^{\ell-1,(1)} - Z_n^{\ell-1,(1)}\|^2]^{\frac12} + \Delta_n^\ell 
    \end{split}
    \end{equation}
    with the convention $\Delta_n^{0}=\E_0[\|X_n^{0,(1)} - Z_n^{0,(1)}\|^2]^{\frac12}$.
    For each $\ell\ge1$, we compute iteratively
    \begin{align*}
        \Delta_{n+1}^{\ell} &\le \Delta_{n}^\ell + \gamma \E_0[\|\widehat R_n^\ell(X_n^{\ell,(1)})-\widetilde{\widehat R}_n^{\ell-1}(\widetilde X_n^{\ell-1,(1)})-(R_n^{\ell}(Z_n^{\ell,(1)})-R_n^{\ell-1}(\widetilde Z_n^{\ell-1,(1)}))\|^2]^{\frac12}\\
        &\le \Delta_{n}^\ell + \gamma \E_0[\|P_{\xell_n}\nabla\log\left(\frac{\xell_n}{\pi_\ell}\right)(X_n^{\ell,(1)})- P_{\mirrormfell_n}\nabla\log\left(\frac{\mirrormfell_n}{\pi_\ell}\right)(Z_n^{\ell,(1)})\|^2]^{\frac12}\\
        &\quad + \gamma \E_0[\|P_{\xelltildeminus_n}\nabla\log\left(\frac{\xelltildeminus_n}{\pi_{\ell-1}}\right)(\widetilde X_n^{\ell-1,(1)})- P_{\mirrormfelltildeminus_n}\nabla\log\left(\frac{\mirrormfelltildeminus_n}{\pi_{\ell-1}}\right)(\widetilde Z_n^{\ell-1,(1)})\|^2]^{\frac12}\\
                            &\quad+\gamma \E_0[\|P_{\mirrormfell_n}\nabla\log\left(\frac{\mirrormfell_n}{\pi_\ell}\right)(Z_n^{\ell,(1)})-P_{\mirrormfelltildeminus_n}\nabla\log\left(\frac{\mirrormfelltildeminus_n}{\pi_{\ell-1}}\right)(\widetilde Z_n^{\ell-1,(1)}) \\ &\quad - \left(P_{\mfell_n}\nabla\log\left(\frac{\mfell_n}{\pi_\ell}\right)(Z_n^{\ell,(1)})-P_{\mfellminus_n}\nabla\log\left(\frac{\mfellminus_n}{\pi_{\ell-1}}\right)(\widetilde Z_n^{\ell-1,(1)})\right)\|^2]^{\frac12}
    \end{align*}
Using Lemma \ref{lemma:dlip} we obtain
    \begin{align*}
        \Delta_{n+1}^\ell&\le \Delta_n^\ell + \gamma \clip \Bigg(\E_0[\|X_n^{\ell,(1)}-Z_n^{\ell,(1)}\|^2]^{\frac12}+\E_0[\mathcal W_2(\xell_n,\mirrormfell_n)]\\ &\quad + \E_0[\|\widetilde X_n^{\ell-1,(1)}-\widetilde Z_n^{\ell-1,(1)}\|^2]^{\frac12} + \E_0[\mathcal W_2(\xelltildeminus_n,\mirrormfelltildeminus_n)]\Bigg) + \gamma \E_0[A]^{\frac12},
    \end{align*}
    where
    \begin{align*} %
      A&:=\Big\|P_{\mirrormfell_n}\nabla\log\left(\frac{\mirrormfell_n}{\pi_\ell}\right)(Z_n^{\ell,(1)})-P_{\mirrormfelltildeminus_n}\nabla\log\left(\frac{\mirrormfelltildeminus_n}{\pi_{\ell-1}}\right)(\widetilde Z_n^{\ell-1,(1)})\\ &\quad  - \left(P_{\mfell_n}\nabla\log\left(\frac{\mfell_n}{\pi_\ell}\right)(Z_n^{\ell,(1)})-P_{\mfellminus_n}\nabla\log\left(\frac{\mfellminus_n}{\pi_{\ell-1}}\right)(\widetilde Z_n^{\ell-1,(1)})\right)\Big\|.
    \end{align*}
  Remember, $\mirrormfell_n$ and $\mirrormfelltildeminus_n$ denote empirical measures over i.i.d.~sample according to $\mfell_n$ and $\mfellminus_n$. Next,
  due to the stationarity of $k$ it holds $k(z,z)=k(0,0)$
    and $\nabla_1 k(z,z)=\nabla_1k(0,0)$ for all $z\in\R^d$.
    Thus
    \begin{align*}
        P_{\mirrormfell_n}\nabla\log\left(\frac{\mirrormfell_n}{\pi_\ell}\right)(Z_n^{\ell,(1)}) &= - \frac{1}{N_\ell} \sum_{m=1}^{N_\ell} \left(\nabla \log\pi_\ell(Z_n^{\ell,(m)})k(Z_n^{\ell,(m)},Z_n^{\ell,(1)}) + \nabla_1 k(Z_n^{\ell,(m)},Z_n^{\ell,(1)})\right)\\ &= -\frac{1}{N_\ell}\nabla \log\pi_\ell(Z_n^{\ell,(1)})k(Z_n^{\ell,(1)},Z_n^{\ell,(1)}) - \frac{1}{N_\ell}\nabla_1 k(Z_n^{\ell,(1)},Z_n^{\ell,(1)})\\ &\quad -\frac{1}{N_\ell} \sum_{m=2}^{N_\ell} \left(\nabla \log\pi_\ell(Z_n^{\ell,(m)})k(Z_n^{\ell,(m)},Z_n^{\ell,(1)}) + \nabla_1 k(Z_n^{\ell,(m)},Z_n^{\ell,(1)})\right)\\
        &= -\frac{1}{N_\ell}\nabla \log\pi_\ell(Z_n^{\ell,(1)})k(z,z)- \frac{1}{N_\ell}\nabla_1 k(z,z)\\ &\quad -\frac{N_\ell-1}{N_\ell} \frac1{N_\ell - 1} \sum_{m=2}^{N_\ell} \left(\nabla \log\pi_\ell(Z_n^{\ell,(m)})k(Z_n^{\ell,(m)},Z_n^{\ell,(1)}) + \nabla_1 k(Z_n^{\ell,(m)},Z_n^{\ell,(1)})\right)
    \end{align*}
    for any $z\in\R^d$. Similarly, one can derive 
    \begin{align*}
        P_{\mirrormfelltildeminus_n}&\nabla\log\left(\frac{\mirrormfelltildeminus_n}{\pi_{\ell-1}}\right)(\widetilde Z_n^{\ell-1,(1)})\\ &= -\frac{1}{N_\ell} \nabla\log\pi_\ell( \widetilde Z_n^{\ell-1,(1)}) k(z,z) - \frac{1}{N_\ell}\nabla_1 k(z,z)\\
         &\quad -\frac{N_\ell-1}{N_\ell} \frac1{N_\ell - 1} \sum_{m=2}^{N_\ell} \left(\nabla \log\pi_{\ell-1}(\widetilde Z_n^{\ell-1,(m)})k(\widetilde Z_n^{\ell-1,(m)},\widetilde Z_n^{\ell-1,(1)}) + \nabla_1 k(\widetilde Z_n^{\ell-1,(m)},\widetilde Z_n^{\ell-1,(1)})\right)\,.
    \end{align*}
    Note that
    $\mathcal W_2(\frac{1}{N_\ell-1} \sum_{m=2}^{N_\ell}
    \delta_{Z_n^{\ell,(m)}},\frac{1}{N_\ell-1} \sum_{m=2}^{N_\ell}
    \delta_{\widetilde Z_n^{\ell-1,(m)}})\lesssim 2^{-\beta \ell}$ by using  \cref{prop:stability_particles}.
    We are ready
    to decompose
    $\mathbb E_0[A]^{\frac12} \le \mathbb E_0[A_1]^{\frac12}+\mathbb
    E_0[A_2]^{\frac12}+\mathbb E_0[A_3]^{\frac12} $ with
    \begin{align*}
        \mathbb E_0[A_1]^{\frac12} &:= \frac{k(z,z)}{N_\ell}  \mathbb E_0[\| \nabla\log\pi_\ell(Z_n^{\ell,(1)})-\nabla\log\pi_{\ell-1}(\widetilde Z_n^{\ell-1,(1)})\|^2]^{\frac12}\\ &\le \frac{k(z,z)}{N_\ell} \left( \mathbb E_0[\| \nabla\log\pi_\ell(Z_n^{\ell,(1)})-\nabla\log\pi_{\ell}(\widetilde Z_n^{\ell-1,(1)})\|^2]^{\frac12} + \mathbb E_0[\| \nabla\log\pi_\ell(\widetilde Z_n^{\ell-1,(1)})-\nabla\log\pi_{\ell-1}(\widetilde Z_n^{\ell-1,(1)})\|^2]^{\frac12} \right)  \\ &\lesssim \frac{2^{-\beta\ell}}{N_\ell}\,,\\
        \mathbb E_0[A_2]^{\frac12} &:= \frac1{N_\ell} \mathbb E_0\Bigg[ \Bigg\| \frac1{N_\ell - 1} \sum_{m=2}^{N_\ell} \left(\nabla \log\pi_\ell(Z_n^{\ell,(m)})k(Z_n^{\ell,(m)},Z_n^{\ell,(1)}) + \nabla_1 k(Z_n^{\ell,(m)},Z_n^{\ell,(1)})\right) \\
        &\quad - \frac1{N_\ell - 1} \sum_{m=2}^{N_\ell} \left(\nabla \log\pi_{\ell-1}(\widetilde Z_n^{\ell-1,(m)})k(\widetilde Z_n^{\ell-1,(m)},\widetilde Z_n^{\ell-1,(1)}) + \nabla_1 k(\widetilde Z_n^{\ell-1,(m)},\widetilde Z_n^{\ell-1,(1)})\right)\Bigg\|^2\Bigg]^{\frac12} \\
        &\lesssim \frac1{N_\ell} \left(\mathbb E_0[\|Z_n^{\ell,(1)}-\widetilde Z_n^{\ell-1,(1)}\|^2]^{\frac{1}2} + \mathcal W_2(\mirrormfell_n,\mirrormfelltildeminus_n)+2^{-\beta\ell} \right) \lesssim \frac{2^{-\beta\ell}}{N_\ell}\,,
    \end{align*}
    where we used Lemma~\ref{lemma:dlip}, and by the Marcinkiewicz-Zygmund inequality \cite[Theorem~5.2]{KM2015} we obtain
    \begin{align*}
        \mathbb E_0[A_3]^{\frac12} &:= \mathbb E_0\Bigg[\Bigg\| \frac1{N_\ell - 1} \sum_{m=2}^{N_\ell} \left(\nabla \log\pi_\ell(Z_n^{\ell,(m)})k(Z_n^{\ell,(m)},Z_n^{\ell,(1)}) + \nabla_1 k(Z_n^{\ell,(m)},Z_n^{\ell,(1)})\right)\\ &\quad-\frac1{N_\ell - 1} \sum_{m=2}^{N_\ell} \left(\nabla \log\pi_\ell(Z_n^{\ell,(m)})k(Z_n^{\ell,(m)},Z_n^{\ell,(1)}) + \nabla_1 k(Z_n^{\ell,(m)},Z_n^{\ell,(1)})\right)\\
        &\quad-\left(P_{\mfell_n}\nabla\log\left(\frac{\mfell_n}{\pi_\ell}\right)(Z_n^{\ell,(1)})-P_{\mfellminus_n}\nabla\log\left(\frac{\mfellminus_n}{\pi_{\ell-1}}\right)(\widetilde Z_n^{\ell-1,(1)})\right)\Bigg\|^2\Bigg]^{\frac12}\\
        &\le \frac{1}{\sqrt{N_\ell-1}} \left(\E_0[\|Z_n^{\ell,(1)}-\widetilde Z_n^{\ell-1,(1)}\|^2]^{\frac12} + \mathcal W_2(\mfell_n,\mfellminus_n)\right) \lesssim \frac{2^{-\beta \ell}}{\sqrt{N_\ell}}\, ,
    \end{align*}
    since $\{Z_n^{\ell,(m)}\}_{m=2}^{N_\ell}$ and $\{\widetilde Z_n^{\ell-1,(m)}\}_{m=2}^{N_\ell}$ are i.i.d.~samples according to $\mfell_n$ and $\mfellminus_n$.

    Using the nested behavior \eqref{eq:nested} and the additional bounds 
    \begin{align*}
    \E_0[\mathcal W_2(\xell_n,\mirrormfell_n)] &\le \E_0[\|X_n^{\ell,(1)}-Z_n^{\ell,(1)}\|^2]^{\frac12}\le \Delta_n^{\ell-1} + \E_0[\|X_n^{\ell-1,(1)}-Z_n^{\ell-1,(1)}\|^2]^{\frac12} +\Delta_n^\ell, \\ \E_0[\mathcal W_2(\xelltildeminus_n,\mirrormfelltildeminus_n)] &\le \E_0[\|\widetilde X_n^{\ell-1,(1)}-\widetilde Z_n^{\ell-1,(1)}\|^2]^{\frac12}\le \E_0[\|X_n^{\ell-1,(1)}-Z_n^{\ell-1,(1)}\|^2]^{\frac12} +\Delta_n^\ell\, ,
    \end{align*}
    we obtain the iterative bound for $(\Delta_n^\ell)_{n\ge0}$ written as
    \begin{align}
        \Delta_{n+1}^\ell &\le \Delta_n^{\ell} + \gamma \clip(\Delta_n^{\ell-1}+2\E_0[\|X_n^{\ell-1,(1)}-Z_n^{\ell-1,(1)}\|^2]^{\frac12} + 2\Delta_n^\ell) + C \frac{2^{-\beta \ell}}{\sqrt{N_\ell}} \notag \\
        &\le (1+2\gamma\clip) \Delta_n^\ell+\gamma\clip \Delta_n^{\ell-1} + C\left( \frac{1}{\sqrt{N_0}} + \sum_{m=1}^\ell \frac{2^{-\beta m}}{\sqrt{N_m}}\right). \label{eq:delta_bound}
    \end{align}
    Here, we have used the induction %
    hypothesis
    \eqref{eq:ind_req}.

    We will use another (nested) inductive argument over $\ell' = 0,\dots,\ell$ to verify that
    \begin{equation}\label{eq:induction2}
      \Delta_j^{\ell'} \lesssim \frac{1}{\sqrt{N_0}} + \sum_{m=1}^{\ell'} \frac{2^{-\beta m}}{\sqrt{N_m}}\qquad\text{for all } j\le n. %
    \end{equation}
    For $\ell'=0$, the bound $\Delta_j^0\lesssim \frac{1}{\sqrt{N_0}}$ holds by \eqref{eq:induction_start}.
    Next, suppose the induction hypothesis is true for some $\ell'-1\ge 0$, i.e.~$\Delta_j^{\ell'-1} \lesssim \frac{1}{\sqrt{N_0}} + \sum_{m=1}^{\ell'-1} \frac{2^{-\beta m}}{\sqrt{N_m}}$. %
    Then we deduce from \eqref{eq:delta_bound} that
    \begin{align*}
        \Delta_{j+1}^{\ell'} 
        &\lesssim (1+2\gamma\clip) \Delta_{j}^{\ell'}+\gamma\clip \sum_{m=1}^{\ell'-1} \frac{2^{-\beta m}}{\sqrt{N_m}} + C \left( \frac{1}{\sqrt{N_0}} + \sum_{m=1}^{\ell'} \frac{2^{-\beta m}}{\sqrt{N_m}}\right).
    \end{align*}
    By the discrete Gronwall inequality, see Lemma~\ref{lem:discr_Gr}, with $\Delta_0^{\ell'} = 0$ we obtain that $\Delta_j^\ell \lesssim \frac{1}{\sqrt{N_0}} + \sum_{m=1}^{\ell'} \frac{2^{-\beta m}}{\sqrt{N_m}}$ for all $j\le n$, which shows \eqref{eq:induction2}. Finally, we %
    obtain with \eqref{eq:nested}
    \[\E_0[\|X_n^{\ell,(1)}-Z_n^{\ell,(1)}\|^2]^{\frac12} %
      \lesssim \frac{1}{\sqrt{N_0}} + \sum_{m=1}^\ell \frac{2^{-\beta m}}{\sqrt{N_m}},   \]
    which concludes the proof of the lemma.
\end{proof}

\section{Proofs of \Cref{sec:complexity}}

\begin{proof}[Proof of \Cref{thm:complexity_SL}]
    Firstly, for the choice $L= \ceil{\frac{1}{\beta\log(2)}\log(\frac2\varepsilon)}$ we obtain 
    $2^{-\beta L} \le \varepsilon/2$ and with $N_L \propto \varepsilon^{-2}$ the expected error is bounded by
    \[\E_0[|\xellL_n[\varphi]- \mfsingle_n[\varphi]|^2]^{\frac12} \lesssim \frac{1}{\sqrt{N_L}} + 2^{-\beta L} \lesssim \varepsilon.\]
    The resulting computational cost is
    \[\cost_{\SL} = n  N_L  2^{qL} = n N_L 2^{\frac{q}{\beta \log(2)}\log(\frac2\varepsilon)} \lesssim \varepsilon^{-2-\frac{q}{\beta}}.\]
\end{proof}

\begin{proof}[Proof of \Cref{thm:complexity_ML}]
  By
    Theorem~\ref{thm:err_ML} the expected error of the ML estimator is bounded by
    \[\E_0[|\widehat\mfsingle_n^{\ML}[\varphi]- \mfsingle_n[\varphi]|^2]^{\frac12} \lesssim (L+1)\left(\frac{1}{\sqrt{N_0}} + \sum_{\ell=1}^L \frac{2^{-\beta\ell}}{\sqrt{N_\ell}}\right) + 2^{-\beta L},\]
    where with $L= \ceil{\frac{1}{\beta\log(2)}\log(\frac2\varepsilon)}$ we have that $2^{-\beta L} \le \varepsilon/2$. With the choice $N_\ell \propto (L+1)^4 2^{-2\beta\ell}\varepsilon^{-2}$ we obtain
    $\frac{2^{-\beta\ell}}{\sqrt{N_\ell}} \lesssim \frac{1}{(L+1)^2}\varepsilon$
    and therefore 
    \[\E_0[|\widehat\mfsingle_n^{\ML}[\varphi]- \mfsingle_n[\varphi]|^2]^{\frac12} \lesssim (L+1)\sum_{\ell=0}^L\frac{\varepsilon}{(L+1)^2} + 2^{-\beta L} \lesssim \varepsilon.\]
    The computational cost is given by
    \[\cost_{\ML} = n \left(\sum_{\ell = 0}^L N_\ell  2^{q\ell} + \sum_{\ell = 1}^L N_\ell  2^{q(\ell-1)} \right) \lesssim n 2^{1-q}(L+1)^4\varepsilon^{-2} \sum_{\ell = 0}^L 2^{(q-2\beta)\ell}. \]
    For $q = 2\beta$ it holds $\sum_{\ell = 0}^L 2^{(q-2\beta)\ell} = L+1$, which yields \[\cost_{\ML}\lesssim n2^{1-q}(L+1)^5 \varepsilon^{-2}\lesssim |\log(\varepsilon)|^5 \varepsilon^{-2}.\] For $q<2\beta$ it holds $\sum_{\ell = 0}^L 2^{(q-2\beta)\ell}\le 2$ such that \[\cost_{\ML}\lesssim n2^{1-q}(L+1)^4 \varepsilon^{-2}\lesssim |\log(\varepsilon)|^4 \varepsilon^{-2}.\] Finally, for $q>2\beta$ we have $\sum_{\ell = 0}^L n2^{(q-2\beta)\ell}= \frac{2^{(q-\beta)(L+1)}-1}{2-1} \le 2^{(q-2\beta)(L+1)}\lesssim \varepsilon^{(\frac{q}{\beta}-2)}$ and the total cost is bounded by \[\cost_{\ML}\lesssim 2^{1-q} (L+1)^4 \varepsilon^{-2-(\frac{q}{\beta}-2)}\lesssim |\log(\varepsilon)|^4 \varepsilon^{-\frac{q}{\beta}}.\]
\end{proof}

\end{document}